\documentclass[11pt,a4paper]{article}

\usepackage[margin=1in]{geometry}
\usepackage[utf8]{inputenc}
\usepackage[T1]{fontenc}
\usepackage{amsmath}
\usepackage{enumerate}
\usepackage{amssymb}
\usepackage{ amsthm} 
\usepackage{graphicx}
\usepackage{MnSymbol}
\usepackage{accents}

\usepackage{tikz-cd}
\usepackage{tkz-graph}
\usetikzlibrary{decorations.markings}
\usepackage{hyperref}
\usepackage{bigints}
\hypersetup{colorlinks, linkcolor=blue}
\usepackage{cleveref}
\usepackage[stable]{footmisc}

\DeclareMathOperator{\Ree}{Re}
\DeclareMathOperator{\Imm}{Im}
\renewcommand{\Re}{{\Ree}}

\newcommand{\Li}{\text{Li}}

\theoremstyle{definition}
\newtheorem{defi}{Definition}
\theoremstyle{definition}
\newtheorem{rmk}{Remark}
\theoremstyle{definition}

\theoremstyle{plain}
\newtheorem{thm}{Theorem}
\theoremstyle{plain}
\newtheorem{prop}{Proposition}
\theoremstyle{plain}
\newtheorem{lemma}{Lemma}
\theoremstyle{plain}
\newtheorem{cor}{Corollary}

\newcommand{\Q}{\mathbb{Q}}

\newcommand{\R}{\mathbb{R}}
\newcommand{\Cc}{\mathbb{C}}

\newcommand{\Pp}{\mathbb{P}}

\newcommand{\Sp}{\operatorname{Spec}}

\newcommand{\coker}{\operatorname{coker}}

\newcommand{\rk}{\operatorname{rk}}

\newcommand{\ra}{\rightarrow}

\title{\textbf{Motivic Galois coaction and one-loop Feynman graphs}}
\author{Matija Tapušković}
\date{}

\begin{document}
\maketitle

\begin{abstract}
 Following the work of Brown, we can canonically associate a family of motivic periods -- called the motivic Feynman amplitude -- to any convergent Feynman integral, viewed as a function of the kinematic variables. The motivic Galois theory of motivic Feynman amplitudes provides an organizing principle, as well as strong constraints, on the space of amplitudes in general, via Brown's "small graphs principle". This serves as motivation for explicitly computing the motivic Galois action, or, dually, the coaction of the Hopf algebra of functions on the motivic Galois group. In this paper, we study the motivic Galois coaction on the motivic Feynman amplitudes associated to one-loop Feynman graphs. We study the associated variations of mixed Hodge structures, and provide an explicit formula for the coaction on the four-edge cycle graph -- the box graph -- with non-vanishing generic kinematics, which leads to a formula for all one-loop graphs with non-vanishing generic kinematics in four-dimensional space-time. We also show how one computes the coaction in some degenerate configurations -- when defining the motive of the graph requires blowing up the underlying family of varieties -- on the example of the three-edge cycle graph.
\end{abstract}

\begin{section}*{Introduction}
\begin{subsection}{Context}

Integrals we will be interested in are those of the form:

\begin{equation}
\label{first-integral}
\int_{\sigma_G} \omega_G(m,q) ,
\end{equation}
where
\begin{equation}
\omega_G(m,q) = \frac{1}{\Psi^{d/2}_G}\left(\frac{\Psi_G}{\Xi_G}\right)^{N_G - h_G d/2} \Omega_G ,
\end{equation}
$G$ is a Feynman graph, $\Psi_G$, and $\Xi_G(m,q)$ are certain polynomials in the variables $\alpha_e$ indexed by the edges of $G$, $d,N_G,$ and $h_G$ are fixed integers, $\Omega_G = \sum\limits_{i=1}^{N_G} (-1)^i\alpha_id\alpha_1 \wedge \ldots \wedge \widehat{d\alpha_i}\wedge \ldots \wedge d\alpha_{N_G}$, and the domain of integration $\sigma$ is given by the real points of the coordinate simplex $\alpha_e \geq 0$. The polynomial $\Xi_G(m,q)$ depends on kinematic parameters. For algebraic values of those parameters these integrals are, when they converge, \textit{periods} in the sense of Kontsevich-Zagier \cite{KZ}. Up to a factor, which is a special value of the gamma function, the integral \eqref{first-integral} is the Feynman integral associated to the graph $G$, in parametric form.

The fact that Feynman integrals can be viewed as periods enables a motivic approach to studying interesting patterns in their evaluations, such as patterns of multiple zeta values studied by Bloch, Esnault and Kreimer in \cite{BEK}. When viewed in this light, a new structure satisfied by Feynman integrals arises, as all periods conjecturally carry an action of the motivic Galois group \cite{Y}. Further study of this structure led to a very striking \textit{coaction conjecture} \cite[Conjecture 1.3]{PS}, which states that the action of the motivic Galois group should be closed on motivic Feynman amplitudes of primitive log-divergent graphs in $\phi^4$ theory, and moreover that the Galois conjugates of such Feynman integrals should be periods associated to subquotient graphs if one allows for non-$\phi^4$ primitive log-divergent graphs. The coaction conjecture is checked numerically therein for hundreds of examples, some of which have 11 loops. This, along with other evidence in different theories, e.g. \cite{S}, leads one to speculate that such a structure might exist more generally, possibly after enlarging the space of periods under consideration appropriately. An important reason for studying this structure is that any results of this type combined with easy results for small graphs provide very strong constraints on Feynman integrals to all loop orders. This is referred to as \textit{the small graphs principle} \cite[8.4,9.3]{Brown2}.

We will be working in the category of realizations over a smooth base scheme over $\mathbb{Q}$, following \cite[\S 1.21]{DelGFD} and \cite[\S 7.2]{Brown2}. In order to study the motivic Galois coaction we must first lift the Feynman integrals to \textit{motivic periods} defined as functions on the scheme of isomorphisms between the de Rham and Betti fiber functors of the category of realizations. Moreover, we would like to work in a more general context than the one in \cite{BEK,PS} by considering Feynman integrals as functions of masses and momenta of particles, which leads us to consider \textit{families of motivic periods}. In \cite{Brown2}, Brown provides a lifting of Feynman integrals to families of motivic periods, with mild constraints on the possible kinematics, thereby setting up the prerequisites for studying the Galois theory of a very general class of Feynman integrals. He also explains why we expect the Galois conjugates of motivic lifts of Feynman integrals to be motivic periods associated to subquotient graphs \cite[Conjecture 1]{Brown2}, and proves this in the "affine case" \cite[Theorem 8.11]{Brown2}.

In order to go further in this direction, we would like to understand in detail the  Galois theory of a family of Feynman integrals where we allow masses and momenta.  It is a theorem due to Nickel \cite{N} that one-loop integrals in four-dimensional space-time always evaluate to linear combinations of integrals associated to one-loop graphs with four edges, which in turn can be evaluated in terms of dilogarithms. From the perspective of algebraic geometry this was studied in \cite{BK}. It is shown there that the fact that these integrals evaluate to dilogarithms is a consequence of the fact that the geometry underlying these integrals carries a mixed Tate Hodge structure with weights 0, 2, and 4, which varies in a family over the space of kinematics. These structures are very well understood in algebraic geometry, and we use this here to study the Galois theory of one-loop integrals with kinematic dependence.
\end{subsection}

\begin{subsection}{Contents}

In the first section we provide a brief overview of the technical background and terminology. We recall the definition of families of motivic periods which we will be working with, as well as the definition of families of de Rham periods, and the de Rham Galois group. The results will be stated in terms of the coaction of the Hopf algebra of functions on the Galois group. We also briefly recall the definition of the \textit{motivic Feynman amplitude} $I_G^\mathfrak{m}(m,q)$, which is the family of motivic periods associated to the Feynman integral $I_G(m,q)$ of the Feynman graph $G$, depending on internal masses and external momenta. We use the term "motivic Feynman amplitude" for the motivic lift of a Feynman integral following Brown \cite{Brown2}. We will often drop $(m,q)$ altogether in order to ease notation, but dependence on masses and momenta is assumed throughout the paper. The second section contains general results on the (realizations of) motives of one-loop graphs, and can be regarded as restating the results of \cite{BK} in terms of motivic periods. In particular it contains the reduction of the motivic Feynman amplitude of any graph with more than four edges to a $k_S$-linear combination of motivic Feynman amplitudes of four-edge graphs in four space-time dimensions, which amounts to lifting the result of Nickel to motivic periods. We also recall how one computes the semi-simplifications of the associated mixed Hodge structures. In the third section we compute the coaction on the motivic Feynman amplitude on the four-edge one-loop graph with non-vanishing generic kinematics, which gives us the coaction for any one-loop graph with non-vanishing generic kinematics by the results of the previous section. Bearing in mind the definition of the associated motivic Feynman amplitude $I^\mathfrak{m}_G$ and its de Rham version $I^{\mathfrak{dr}}_G$ \eqref{motivic-Feynman-amplitude}, the definition of the Galois coaction \eqref{general-coaction}, as well as the definitions of the de Rham logarithm $\log^{\mathfrak{dr}}(x)$ and the Lefschetz de Rham period $\mathbb{L}^{\mathfrak{dr}}$ \ref{examples}, the main result of the third section is the following:
\begin{thm}
Let $G$ be a one-loop Feynman graph with 4 edges (Figure \ref{box-graph}), with generic non-vanishing masses and momenta. Let $I^\mathfrak{m}_G$ be its associated motivic Feynman amplitude in $d=4$ dimensions. Then the motivic Galois coaction on $I^\mathfrak{m}_G$ is:
\begin{equation}
\begin{split}
\Delta I_G^{\mathfrak{m}} &= I_G^{\mathfrak{m}} \otimes \left(\mathbb{L}^{\mathfrak{dr}}\right)^2 + \sum_{1 \leq j < k \leq 4}  I^{\mathfrak{m}}_{G/\{e_j,e_k\}}(\theta^1_{G/\{e_j,e_k\}})\otimes P_{j,k}\log^{\mathfrak{dr}}(f_{j,k}) \mathbb{L}^{\mathfrak{dr}}+ 1 \otimes I^{\mathfrak{dr}}_G ,
\end{split}
\end{equation}
where $I^{\mathfrak{m}}_{G/\{e_j,e_k\}}\left(\theta^1_{G/\{e_j,e_k\}}\right)$ is the motivic Feynman amplitude of the bubble graph obtained by contracting the edges $e_j$ and $e_k$ of $G$ in $d=2$ dimensions, 
\begin{equation}
f_{j,k} = \frac{\sqrt{(U)_{j,k}^2 - (U)_{j,j}(U)_{k,k}}-(U)_{j,k}}{\sqrt{(U)_{j,k}^2 - (U)_{j,j}(U)_{k,k}}+(U)_{j,k}},
\end{equation}
where $C,D_{j,k}$ are the matrices of the quadratic forms $\Xi_G$, $\Xi_{G/\{e_j,e_k\}}$ respectively, $U = C^{-1}$, \eqref{Symanzik-four-edge}, and $P_{j,k}=\frac{\sqrt{|\det D_{j,k} |}}{8\sqrt{|\det C|}}$.
\end{thm}
The motivic periods in the coaction in Theorem 1 are identified with motivic periods of quotient graphs via certain natural homomorphisms in the category of realizations called \textit{the face maps}, while the de Rham periods are determined by using the Gauss-Manin connection. In the next section, in which we study the triangle graph, another approach is used to determine the de Rham periods -- namely the residue homomorphism.

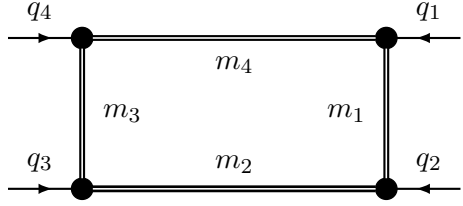
\begin{figure}[h]
\centering
\begin{tikzpicture}

\SetGraphUnit{1}
  
  \SetUpEdge[lw = 1pt,
  color      = black,
  labelcolor = white,
  labelstyle = {above,yshift=2pt}]
  
  \SetUpVertex[FillColor=black, MinSize=8pt, NoLabel]

  \Vertex[x=2,y=2]{0}
  \Vertex[x=2,y=0]{1}
  \Vertex[x=6,y=0]{2}
  \Vertex[x=6,y=2]{6}
  \Vertex[x=1,y=2,empty=true]{3}
  \Vertex[x=1,y=0,empty=true]{4}
  \Vertex[x=7,y=0,empty=true]{5}
  \Vertex[x=7,y=2,empty=true]{7}

  \Edge[label=$m_2$,style = {double}](2)(1) 
  \SetUpEdge[labelstyle = {yshift=-10pt}]
  \Edge[label=$m_4$,style = {double}](0)(6)
  \SetUpEdge[labelstyle = {yshift=10pt}]
  \tikzset{EdgeStyle/.style={postaction=decorate,decoration={markings,mark=at position 0.7 with {\arrow{latex}}}}}
  \Edge[label=$q_4$](3)(0)
  \Edge[label=$q_3$](4)(1)
  \Edge[label=$q_2$](5)(2)
  \Edge[label=$q_1$](7)(6)
  
  \SetUpEdge[labelstyle = {xshift=15pt}]
  \Edge[label=$m_3$,style = {inner sep=5pt, double, rotate=90}](0)(1)
  \SetUpEdge[labelstyle = {xshift=-15pt}]
  \Edge[label=$m_1$,style = {double, rotate=90}](2)(6) 
\end{tikzpicture}
\caption{The box graph. Internal edges which have non-vanishing mass are denoted by double lines.}
\label{box-graph}
\end{figure}

The coaction commutes with specialization to a point in an open subset of the space of kinematics over which the graph motive is defined. For closed subsets outside of the space of generic kinematics, such as those given by the vanishing of some masses, to which Theorem 1 therefore does not apply, one can still apply the methods presented in this paper to compute the Galois coaction, so long as the values of masses and momenta are such that $I_G(m,q)$ converges. However, one should bear in mind that in such cases it might not be possible to interpret the conjugates in the coaction in terms of motivic periods of motives of subquotient graphs. Understanding this subtlety motivates a more detailed study of a couple of special cases in the next section, in particular Theorem 2.

In the fourth section we will concern ourselves with the three-edge graph. The result for the case with generic non-vanishing kinematics is quite similar to the box graph case above, but a few special cases when some of the masses vanish reveal subtleties. The result in one of these special cases is the following:
\begin{thm}
Let $G$ be the Feynman graph with 3 edges and 1 loop, with all internal masses vanishing and non-trivial external momenta. Denote the three external momenta by $q_1,q_2,q_3$, and the associated motivic Feynman amplitude by $I^\mathfrak{m}_G$. Then the coaction on this motivic Feynman amplitude is:
\begin{equation}
\Delta I_G^{\mathfrak{m}} = I_G^{\mathfrak{m}} \otimes  \left(\mathbb{L}^{\mathfrak{dr}}\right)^2 + \left( a_1 \log^{\mathfrak{m}}\left(\frac{q_2^2}{q_3^2}\right) + a_2 \log^{\mathfrak{m}}\left(\frac{q_1^2}{q_3^2} \right)\right) \otimes \left(\log^{\mathfrak{dr}}(f_1) + \log^{\mathfrak{dr}}(f_2)\right)\mathbb{L}^{\mathfrak{dr}} + 1 \otimes I^{\mathfrak{dr}}_G ,
\end{equation}
where
\[
f_1= \frac{(q_1^2 + q_2^2-q_3^2 + \sqrt{q_1^4+q_2^4+q_3^4-2q_1^2q_3^2-2q_2^2q_3^2})^2}{4q_1^2q_2^2}
\]
and 
\[
f_2 = f_1 \frac{q_1^2 + q_3^2-q_2^2 - \sqrt{q_1^4+q_2^4+q_3^4-2q_1^2q_3^2-2q_2^2q_3^2}}{q_1^2 + q_3^2-q_2^2 + \sqrt{q_1^4+q_2^4+q_3^4-2q_1^2q_3^2-2q_2^2q_3^2}}
\]
for some undetermined constants $a_i \in k_S$.
\end{thm}

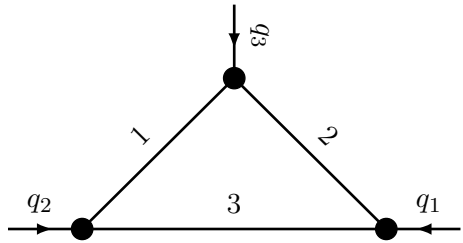
\begin{figure}[h]
    \centering
    \begin{tikzpicture}
    \SetGraphUnit{1}
  
    \SetUpEdge[lw = 1pt,
  color      = black,
  labelcolor = white,
  labelstyle = {sloped,above,yshift=2pt}]
  
  \SetUpVertex[FillColor=black, MinSize=8pt, NoLabel]

  \Vertex[x=4,y=2]{0}
  \Vertex[x=2,y=0]{1}
  \Vertex[x=6,y=0]{2}
  \Vertex[x=4,y=3,empty=true]{3}
  \Vertex[x=1,y=0,empty=true]{4}
  \Vertex[x=7,y=0,empty=true]{5}

  \Edge[label=1](0)(1)
  \Edge[label=2](0)(2)
  \Edge[label=3](2)(1)
  \tikzset{EdgeStyle/.style={postaction=decorate,decoration={markings,mark=at position 0.7 with {\arrow{latex}}}}}
  
  \Edge[label=$q_3$](3)(0)
  \Edge[label=$q_2$](4)(1)
  \Edge[label=$q_1$](5)(2)
 
    \end{tikzpicture}
    \caption{The triangle graph. Internal edges denoted with a single line have vanishing mass.}
    \label{fig:triangle}
\end{figure}

In the case of the previous theorem one cannot straightforwardly think of the motivic logarithms appearing in the coaction as motivic periods of the motives of quotient graphs of the triangle graph with all internal masses vanishing. This is because Feynman integrals associated to the one-loop graphs with two edges obtained by contracting an edge are divergent. In order to relate the motivic logarithms in the coaction in Theorem 2 to subquotient graphs, we must consider regularized versions of motivic periods of motives of subquotient graphs \cite[Conjecture 1.]{Brown2}. Alternatively, an approach taken in \cite{Brown2} is to associate to each graph a larger motive, called the \textit{affine motive} \cite[\S 8.5]{Brown2}, which is then proved to capture the motivic periods in the coaction applied to the motivic lift of the Feynman integral associated to the given graph \cite[Theorem 8.11]{Brown2}.

We note that we aim to make the computation of the coaction involving masses and momenta as explicit as possible. For that reason some moderately lengthy computations and bulky notation can be found in the second half of this paper.

\begin{subsubsection}*{Acknowledgements}
This project has received funding from the European Research Council (ERC) under the European Union’s Horizon 2020 research and innovation programme (grant agreement No. 724638). The author owes special thanks to Francis Brown for his encouragement to study the motivic Galois theory of Feynman periods in general, and many discussions and comments on this paper in particular. The author also thanks Erik Panzer for discussions, and in particular for his help with a computation in section 3 which compares the computations in this paper with those in the physics literature. Thanks are also owed to Clément Dupont and Nils Matthes for helpful discussions.
\end{subsubsection}
\end{subsection}

\end{section}
\begin{section}{Brief overview of the technical background}
We recall the definitions, constructions and main results on motivic periods and motivic Feynman amplitudes that we use throughout the article. This is a summary of the results of Brown \cite{Brown1}, \cite{Brown2}.
\begin{subsection}{Category of realizations}
\label{cat-of-realizations-and-mot-periods}

Let $S$ be a smooth geometrically connected scheme over $\mathbb{Q}$. And let $\mathcal{H}(S)$ be the category of triples $(\mathbb{V}_{\rm{B}},(\mathcal{V}_{\mathrm{dR}},\nabla),c)$, where $\mathbb{V}_{\rm{B}}$ is a local system of finite dimensional $\mathbb{Q}$-vector spaces on $S(\mathbb{C})$, and $(\mathcal{V}_{\mathrm{dR}},\nabla)$ is an algebraic vector bundle on $S$ with a flat connection with regular singularities at infinity. Note that for the definition of regular singularities at infinity one needs a good compactification of $S$, but it does not depend on the choice of a good compactification \cite{Katz}. Furthermore $\mathbb{V}_{\rm{B}}$ is equipped with an increasing \textit{weight} filtration $W_{\bullet}\mathbb{V}_{\rm{B}}$ of local sub-systems, and $(\mathcal{V}_{\mathrm{dR}},\nabla)$ is equipped by an increasing \textit{weight} filtration by algebraic sub-vector bundles with a flat connection with regular singularities at infinity $W_{\bullet}\mathcal{V}_{\mathrm{dR}}$ and a decreasing \textit{Hodge} filtration by algebraic sub-vector bundles $F^{\bullet}\mathcal{V}_{\mathrm{dR}}$ which satisfy Griffiths transversality
\[
\nabla : F^p\mathcal{V}_{\mathrm{dR}} \rightarrow F^{p-1}\mathcal{V}_{\mathrm{dR}} \otimes_{\mathcal{O}_S} \Omega^1_{S/k} .
\]
Finally, $c: (\mathcal{V}_{\mathrm{dR}},\nabla)^{an} \xrightarrow{\sim} \mathbb{V}_{\rm{B}} \otimes_{\mathbb{Q}} \mathcal{O}_{S^{an}}$ is an isomorphism of analytic vector bundles with connection which respects the weight filtration, where the connection on the right hand side is one for which the sections of $\mathbb{V}_{\rm{B}}$ are flat. We further require that $\mathbb{V}_{\rm{B}}$, equipped with the filtrations $W$ and $cF$, is a graded-polarizable variation of mixed Hodge structures. The morphisms in $\mathcal{H}(S)$ respect the above data.

The category $\mathcal{H}(S)$ is Tannakian, see \cite{Deligne-Tannakian}. It is equipped with the following fiber functors. For any simply connected $Z \subset S(\mathbb{C})$ let:
\[
\omega^Z_{B} : \mathcal{H}(S) \rightarrow \mathrm{Vec}_{\mathbb{Q}}
\]
be the fiber functor which sends a triple $(\mathbb{V}_{\rm{B}},(\mathcal{V}_{\mathrm{dR}},\nabla),c)$ to the $\mathbb{Q}$-vector space of sections $\Gamma(Z,\mathbb{V}_{\rm{B}})$. We also consider the fiber functor 
\[
\omega^{gen}_{\mathrm{dR}} : \mathcal{H}(S) \rightarrow \mathrm{Vec}_{k_S}
\]
where $k_S$ is the function field of $S$. It sends the vector bundle $(\mathcal{V}_{\mathrm{dR}},\nabla)$ to its fiber over the generic point of $S$.

\end{subsection}
\begin{subsection}{Families of \texorpdfstring{$\mathcal{H}(S)$}{H(S)}-periods}
The \textit{ring of $\mathcal{H}(S)$-periods} is defined as:
\[
\mathcal{P}^{\mathfrak{m},Z,gen}_{\mathcal{H}(S)} := \mathcal{O}(\textrm{Isom}^{\otimes}_{\mathcal{H}(S)}(\omega^{gen}_{\mathrm{dR}},\omega^Z_B)).
\]
It is a $\mathbb{Q}, k_S$-bimodule generated by equivalence classes of triples $((\mathbb{V}_{\rm{B}},(\mathcal{V}_{\mathrm{dR}},\nabla),c), \sigma,\omega)$, where $\mathcal{V}:=(\mathbb{V}_{\rm{B}},(\mathcal{V}_{\mathrm{dR}},\nabla),c) \in \textrm{ob}(\mathcal{H}(S))$, $\sigma \in (\omega^Z_B\mathcal{V})^{\vee}$, and $\omega \in \omega^{gen}_{\mathrm{dR}}\mathcal{V}$, modulo the relations:
\begin{equation}
    \begin{split}
        & (\mathcal{V},\lambda_1 \sigma_1 + \lambda_2 \sigma_2, \omega) \sim \lambda_1(\mathcal{V},\sigma_1, \omega) + \lambda_2 (\mathcal{V},\sigma_2, \omega) \\
        & (\mathcal{V},\sigma, \lambda_1\omega_1 + \lambda_2\omega_2) \sim \lambda_1(\mathcal{V},\sigma, \omega_1) + \lambda_2 (\mathcal{V},\sigma, \omega_2),
    \end{split}
\end{equation}
where $\lambda_1 \in \mathbb{Q}, \lambda_2 \in k_S$. Furthermore, for any morphism $\phi: \mathcal{V}_1 \rightarrow \mathcal{V}_2$ in the category $\mathcal{H}(S)$ we have:
\[
(\mathcal{V}_1,(\omega^{X}_B(\phi))^t(\sigma_2), \omega_1) \sim (\mathcal{V}_2,\sigma_2, \omega^{gen}_{\mathrm{dR}}(\phi)(\omega_1)),
\]
where $(\omega^{X}_B(\phi))^t$ denotes the dual morphism to $\omega^{X}_B(\phi)$. We denote the equivalence classes of triples by $[(\mathbb{V}_{\rm{B}},(\mathcal{V}_{\mathrm{dR}},\nabla),c), \sigma,\omega]^{\mathfrak{m}}$.

The ring $\mathcal{P}^{\mathfrak{m},Z,gen}_{\mathcal{H}(S)}$ is equipped with a \textit{period homomorphism}:
\begin{equation}
    \begin{split}
        \textrm{per}:\quad& \mathcal{P}^{\mathfrak{m},Z,gen}_{\mathcal{H}(S)} \rightarrow M_Z(S(\mathbb{C})) ,
    \end{split}
\end{equation}
where $M_Z(S(\mathbb{C}))$ denotes the ring of multivalued meromorphic functions on $S(\mathbb{C})$ with a prescribed branch along $Z$. To define it let $\pi: \widetilde{S}(\mathbb{C})_Z \rightarrow S(\mathbb{C})$ be the universal cover of $S(\mathbb{C})$ based at $Z$ (recall that $Z$ is simply connected). Since $\widetilde{S}(\mathbb{C})_Z$ is simply connected the local system $\pi^*(\mathbb{V}^{\vee}_{\rm{B}})$ is trivial, and $\sigma \in \Gamma(Z,\mathbb{V}_{\rm{B}}^{\vee})$ extends to a unique global section $\Gamma(\widetilde{S}(\mathbb{C})_Z, \pi^*(\mathbb{V}^{\vee}_{\rm{B}}))$. Let $x \in \widetilde{S}(\mathbb{C}_Z)$, and $N(x)$ be a small enough neighbourhood of $x$ so that the restriction of $\pi$ is an isomorphism. Define a local section $\sigma_x = (\pi|_{N(x)}^{-1})^*\sigma \in \Gamma(\pi(N(x)),\mathbb{V}_{\rm{B}}^{\vee})$.

Since $\omega \in \Gamma(\Sp(k_S),\mathcal{V}_{\mathrm{dR}})$, there exists an non-empty affine open $U \subset S$ such that $\omega \in \Gamma(U,\mathcal{V}_{\mathrm{dR}})$. Let $W \subset S$ be an affine open such that $\pi(x) \in W(\mathbb{C})$, and furthermore we can make $\pi(N(x)) \in W(\mathbb{C})$ by making $N(x)$ smaller if necessary. Since $S$ is irreducible $U \cap W \not = \emptyset$. The section $\omega|_{U \cap W}$ can have poles along $W \setminus U$, but we can 'clear denominators', i.e., there is $\alpha \in \mathcal{O}_W$ such that $\alpha\omega \in \Gamma(W,\mathcal{V}_{\mathrm{dR}})$. By restriction, and passing to the associated analytic vector bundle we can view $\alpha\omega$ as an element in $\Gamma(\pi(N(x)),\mathcal{V}^{an}_{\mathrm{dR}})$. Then the comparison isomorphism $c : \mathcal{V}^{an}_{\mathrm{dR}} \rightarrow \mathbb{V}_{\rm{B}} \otimes_{\mathbb{Q}}\mathcal{O}_{S^{an}}$ gives an element:
\[
\textrm{per}([(\mathbb{V}_{\rm{B}},(\mathcal{V}_{\mathrm{dR}},\nabla),c), \sigma,\alpha\omega]^{\mathfrak{m}}) = \sigma_x(c(\alpha\omega)) \in \Gamma(\pi(N(x)), \mathcal{O}_{S^{an}}),
\]
which can be viewed as a locally analytic function on $N(x)$.The period homomorphism is defined on $[(\mathbb{V}_{\rm{B}},(\mathcal{V}_{\mathrm{dR}},\nabla),c), \sigma,\omega]^{\mathfrak{m}}$ by dividing by the function $\alpha$. It locally has poles along the zeroes of $\alpha$.

This ring comes equipped with a \textit{connection} 
\[
\nabla : \mathcal{P}^{\mathfrak{m},Z,gen}_{\mathcal{H}(S)} \rightarrow  \mathcal{P}^{\mathfrak{m},Z,gen}_{\mathcal{H}(S)} \otimes_{k_S} \Omega_{k_S/k}^1
\]
which acts on families of motivic periods by:
\[
\nabla [\mathcal{V}, [\sigma], [\omega]]^{\mathfrak{m}} = [\mathcal{V}, [\sigma], \nabla [\omega]]^{\mathfrak{m}}.
\]
For details see \cite[\S 7.4]{Brown1}.

\end{subsection}
\begin{subsection}{Families of de Rham periods and coaction}
We define the ring of \textit{de Rham periods} as:
\[
\mathcal{P}^{\mathfrak{dr},gen}_{\mathcal{H}(S)} := \mathcal{O}(\textrm{Aut}^{\otimes}_{\mathcal{H}(S)}(\omega^{gen}_{\mathrm{dR}})).
\]
It is spanned, over $k_S$, by equivalence classes of triples $[(\mathbb{V}_{\rm{B}},(\mathcal{V}_{\mathrm{dR}},\nabla),c), \upsilon,\omega]^{\mathfrak{dr}}$, where $\upsilon \in (\omega^{gen}_{\mathrm{dR}}\mathcal{V})^{\vee}$, and $\omega \in \omega^{gen}_{\mathrm{dR}}\mathcal{V}$, defined analogously to $\mathcal{H}(S)$-periods. Furthermore, $\mathcal{P}^{\mathfrak{dr},Z,gen}_{\mathcal{H}(S)}$ is a Hopf algebra, and the ring $\mathcal{P}^{\mathfrak{m},Z,gen}_{\mathcal{H}(S)}$ has a right \textit{Galois coaction} by $\mathcal{P}^{\mathfrak{dr},Z,gen}_{\mathcal{H}(S)}$:
\[
\Delta : \mathcal{P}^{\mathfrak{m},Z,gen}_{\mathcal{H}(S)} \rightarrow \mathcal{P}^{\mathfrak{m},Z,gen}_{\mathcal{H}(S)} \otimes_{k_S} \mathcal{P}^{\mathfrak{dr},gen}_{\mathcal{H}(S)} ,
\]
given by the formula:
\begin{equation} 
\label{general-coaction}
\Delta [\mathcal{V},\sigma,\omega]^{\mathfrak{m}} = \sum\limits_{e_i} [\mathcal{V},\sigma,e_i]^{\mathfrak{m}} \otimes [\mathcal{V},e_i^{\vee},\omega]^{\mathfrak{dr}}
\end{equation}
where the $\{e_i\}$ is a basis of $\omega^{gen}_{\mathrm{dR}}\mathcal{V}$, and $e_i^{\vee}$ is the dual basis. Dual to the Galois coaction is the left \textit{Galois action} of the affine group scheme given by $G^{\mathfrak{dr},gen}_{\mathcal{H}(S)}:=\Sp\left(\mathcal{P}^{\mathfrak{dr},gen}_{\mathcal{H}(S)}\right)$:
\[
G^{\mathfrak{dr},egn}_{\mathcal{H}(S)} \times \mathcal{P}^{\mathfrak{m},Z,gen}_{\mathcal{H}(S)} \rightarrow \mathcal{P}^{\mathfrak{m},Z,gen}_{\mathcal{H}(S)}
\]
given by:
\begin{equation}
    \begin{split}
        g [\mathcal{V},\sigma,\omega]^{\mathfrak{m}} &= (1 \otimes g)\Delta [\mathcal{V},\sigma,\omega]^{\mathfrak{m}} = \sum\limits_{e_i} [\mathcal{V},\sigma,e_i]^{\mathfrak{m}} \otimes g \left([\mathcal{V},e_i^{\vee},\omega]^{\mathfrak{dr}}\right),
    \end{split}
\end{equation}
where $g \in G^{\mathfrak{dr},gen}_{\mathcal{H}(S)}(R)$, for $R$ a $k_S$-algebra.

\end{subsection}
\begin{subsection}{Motivic periods}
\label{motivic-periods}
We are particularly interested in objects of $\mathcal{H}(S)$ of a prescribed geometric origin. By this we mean the following. Recall that a simple normal crossing divisor is a normal crossing divisor such that each of its irreducible components is smooth. Assume, as before, that $S$ is a smooth geometrically connected scheme over $\mathbb{Q}$, and let $ D \subset X $ be a family of simple normal crossing divisors relative to a smooth morphism $\pi : X \ra S $ which, on the underlying analytic varieties, we assume to be a locally trivial fibration of stratified analytic varieties, for the stratification on $X$ induced by $D$ -- see  \cite{GM} and \cite[Ch. IV]{Pham}. Let $j : X \setminus D \hookrightarrow X$ be the inclusion, and denote by $D_i$ the irreducible components of $D$, for $i \in I$. Denote by $D_J = \cap_{j\in J} D_j$, for $\not 0 \not = J \subset I$, and let $D_{\emptyset} = X$. 

Because $\pi$ is a locally trivial fibration, the sheaf
\begin{equation}
\label{rel-cohom-definition}
H^n_{\rm{B}}(X,D)_{/S} := R^n\pi_*j_!\Q ,
\end{equation}
where $\mathbb{Q}$ is the constant sheaf on $(X \setminus D)(\mathbb{C})$, is a local system over $S(\mathbb{C})$ with its analytic topology. It is computed by the hypercohomology of the complex of sheaves:
\begin{equation}
\label{relative-double-complex-Betti}
\mathbb{Q}_{D_{\bullet}/S} : 
\mathbb{Q}_{X} \rightarrow \bigoplus_{|J|=1} \mathbb{Q}_{D_J} \rightarrow \bigoplus_{|J|=2} \mathbb{Q}_{D_J} \rightarrow ...
\end{equation}
where $\mathbb{Q}_{D_J}$ is the constant sheaf $\mathbb{Q}$ on $D_J(\mathbb{C})$ extended by 0 to $X(\mathbb{C})$.

 Let $\Omega_{D_J/S}^{\bullet}$ denote the sheaf on $X$ which is the direct image of the corresponding sheaves of Kähler differentials on $D_J$, and which vanishes outside of $D_J$. Consider the double complex of sheaves on $X$
\begin{equation}
\label{relative-double-complex}
\Omega_{D_{\bullet}/S}^{\bullet} : \quad \Omega_{X/S}^{\bullet} \rightarrow \bigoplus\limits_{|J|=1} \Omega_{D_J/S}^{\bullet} \rightarrow \ldots \rightarrow \bigoplus\limits_{|J|=d} \Omega_{D_J/S}^{\bullet}
\end{equation}
where $d$ is the relative dimension of $X$ over $S$, and horizontal maps are pullbacks along inclusions $D_{J} \hookrightarrow D_{J\setminus i_j}$ multiplied by $(-1)^k$ if $i_j$ is the $k$th element of $J$. We will denote such restrictions, including the signs, by $r^{J}_{i_j}$ and write simply $r_{i_j}$ when $J = \{i_j\}$. Then define
\[
H^n_{\mathrm{dR}}(X,D)_{/S} =\mathbb{R}^n\pi_*(\textrm{Tot}^{\bullet}(\Omega_{D_{\bullet}/S}^{\bullet})),
\]
where $\textrm{Tot}^{\bullet}$ denotes the total complex. It should have a flat connection
\[
\nabla : H^n_{\mathrm{dR}}(X,D)_{/S} \rightarrow H^n_{\mathrm{dR}}(X,D)_{/S} \otimes \Omega^1_{S/k} 
\]
by a relative version of \cite{KO}.

 By \cite[Proposition 2.28]{Deligne4}, using the assumption that $\pi$ is topologically trivial, we have an isomorphism:
\[
c^{-1} : H^n_{\rm{B}}(X,D)_{/S} \otimes_{\mathbb{Q}} \mathcal{O}^{an}_S \xrightarrow{\sim} (H^n_{\mathrm{dR}}(X,D)_{/S})^{an} .
\]

Another gap in the literature seems to be that $H^n_{\rm{B}}(X,D)_{/S}$, with its weight filtration and a Hodge filtration coming from $cH^n_{\mathrm{dR}}(X,D)_{/S}$,  should be a variation of mixed Hodge structure. Putting everything together, and admitting the above stated claims, we get an element 
\begin{equation}
\label{motivic-objects}
(H^n_{\rm{B}}(X,D)_{/S},H^n_{\mathrm{dR}}(X,D)_{/S},c)
\end{equation}
of the category $\mathcal{H}(S)$, which we denote $H^n(X,D)_{/S}$. We refer to the $\mathcal{H}(S)$-periods associated to such objects of $\mathcal{H}(S)$ as \textit{families of motivic periods}.

For the purposes of this paper we will define explicitly an open $S$ as a complement of a certain closed subset of an irreducible affine algebraic variety over which $H^n_{\rm{B}}(X,D)_{/S}$ will be a local system, and we will work with a family of divisors $D \subset X$ which is simple normal crossing over the generic point of the said irreducible affine algebraic variety. Therefore we will only consider the double complex \eqref{relative-double-complex} over the generic point.   

\begin{defi}
\label{face-map}
Let $D^I = \cup_{j \not \in I} D_j$, and $k=|I|$. \textit{Face maps} are morphisms in the category $\mathcal{H}(S)$
\begin{equation}
    H^{n-k}(D_I,D^I \cap D_I)_{/S} \rightarrow H^n(X,D)_{/S},
\end{equation}
defined by the inclusion of complexes $\Omega^{\bullet - k}_{D^I_{\bullet}/S} \rightarrow \Omega^{\bullet}_{D_{\bullet}/S}$, and $\mathbb{Q}_{D^I_{\bullet}/S}[-k] \rightarrow \mathbb{Q}_{D_{\bullet}/S}$, on the de Rham and Betti realizations respectively.
\end{defi}

\begin{subsubsection}{Mixed Tate Hodge structures}
\label{mixed Tate HS}
In this paper we will focus on a special case of families of motivic periods coming from objects of $\mathcal{H}(S)$ associated to variations of mixed Tate Hodge structures. To make this precise we recall a few definitions. A \textit{Tate Hodge structure} $\mathbb{Q}(-m)$ is the pure Hodge structure of weight $2m$ defined by 
\[
H_{\mathbb{Q}} = (2\pi i)^{-m}\mathbb{Q}, \quad H_{\mathbb{C}} = H^{m,m}(\mathbb{Q}(-m)_{\mathbb{C}}).
\]
We can pull back $\mathbb{Q}(-m)$ to $S$ via the structure map $S \rightarrow \Sp(\mathbb{Q})$ to obtain a constant \textit{variation of Tate structure}, denoted $\mathbb{Q}(-m)_{/S}$.

A mixed Tate Hodge structures are mixed Hodge structures $H$ such that the weight graded quotients are $gr^W_{2m}H = \bigoplus \mathbb{Q}(-m)$ and $gr^W_{2m+1}H = 0$. They are extensions of Tate Hodge structures. We will be working with objects of $\mathcal{H}(S)$ such that their fibers over $S(\mathbb{C})$ are mixed Tate Hodge structures. In this particular case, the choice of terminology where we refer to "motivic periods" is justified by the fact that the Hodge realization functor is fully faithful on mixed Tate motives over number fields \cite{DG}.
\end{subsubsection}

\begin{subsubsection}{Examples}
\label{examples}
Let $S = \Pp^1 \setminus \{0,1,\infty\}$, $X = S \times \mathbb{G}_m$, and $\pi:X \rightarrow S$ the projection. Let $x$ be the coordinate on $S$ and $y$ a coordinate on $\mathbb{G}_m$. Define $D=\{y=1\} \cup \{y=x\}$. We consider the object $\mathcal{V} = (H^1_B(X,D)_{/S},H^1_{\mathrm{dR}}(X,D)_{/S},c) \in \mathcal{H}(S)$.

We can choose $Z$ to be the real interval $(0,1)$ of $S(\mathbb{C})$, and define for all $x\in Z$ a cycle $\sigma_x \in \mathbb{G}_m(\mathbb{C})$ -- a straight line path from $1$ to $x$, which defines a class in $(\omega^Z_{B}(H^1(X,D)_{/S}))^{\vee}$. Note that in this case the differential form $\frac{dy}{y}$ defines a class in $\Gamma(S,(H_{\rm{dR}}^1(X,D)_{/S}))$, and we do not have to restrict to working over the generic point only. We define the motivic logarithm as a family of motivic periods
\[
\log^{\mathfrak{m}}(x) = \left[\mathcal{V},[\sigma_x],\left[\frac{dy}{y}\right]\right]^{\mathfrak{m}} \in \mathcal{P}^{\mathfrak{m},Z,{gen}}_{\mathcal{H}(S)}
\]
Its period is the logarithm
\[
\textrm{per}(\log^{\mathfrak{m}}(x)) = \int_{\sigma_x} \frac{dy}{y} = \log(x)
\]
for $x \in Z$. Later in this paper we will consider logarithms over a higher dimensional, but still irreducible affine, base than $\mathbb{P}^1 \setminus \{ 0,1,\infty \}$, for which case we must extend the definition of the motivic logarithm given above. However note that we will always be given a subset $Z$ of the complex points of our base over which the branch of the logarithm is unambiguous, and the definition is thus extended in an obvious way.

In order to lift $2\pi i$ to its motivic version it is enough to define it over $S = \Sp(\mathbb{Q})$, and it can then be pulled back to a constant family over $S' \rightarrow S$. Consider the object $H=(H^1_{\textrm{B}}(\mathbb{G}_m),H^1_{\mathrm{dR}}(\mathbb{G}_m),c) \in \textrm{ob}(\mathcal{H})$, where $\mathcal{H} = \mathcal{H}(\Sp(\mathbb{Q}))$, and let $\gamma_0$ be a positively oriented circle around 0. We have that $[\gamma_0] \in H^{\vee}_B$, and
\[
\mathrm{per}\left(\left[H,[\gamma_0],\left[\frac{dt}{t}\right]\right]^{\mathfrak{m}}\right) = \int_{\gamma_0} \frac{dt}{t} = 2 \pi i
\]
We denote
\[
\mathbb{L}^\mathfrak{m} = \left[H,[\gamma_0],\left[\frac{dt}{t}\right]\right]^{\mathfrak{m}} \in \mathcal{P}^{\mathfrak{m}}_{\mathcal{H}},
\]
where we have dropped $Z,gen$ from the notation since we are working over a point, and we refer to it as the \textit{Lefschetz motivic period}. We denote the constant family of motivic periods obtained by pulling back $\mathbb{L}^\mathfrak{m}$ to $S$ via its structure morphism by the same symbol.

We will also consider the de Rham verisons of these motivic periods. We choose a basis of the de Rham realization $\left\{\left[\frac{dy}{y}\right],\left[\frac{dy}{1-x}\right]\right\}$, and denote its dual basis $\left\{\left[\frac{dy}{y}\right]^{\vee},\left[\frac{dy}{1-x}\right]^{\vee}\right\}$. For the de Rham logarithm define
\[
\log^{\mathfrak{dr}}(x) = \left[\mathcal{V},\left[\frac{dy}{1-x}\right]^{\vee}, \left[\frac{dy}{y}\right]\right]^{\mathfrak{dr}} .
\]
This definition of the de Rham logarithm is motivated by it being the image of the motivic logarithm, defined earlier, under the \textit{de Rham projection} -- see \cite[\S 4, Example 4.5.1]{BD}. For the Lefschetz de Rham period define:
\begin{equation}
\label{dr-lefschetz}
\mathbb{L}^\mathfrak{dr}:= \left[H,\left[\frac{dt}{t}\right]^{\vee},\left[\frac{dt}{t}\right]\right]^{\mathfrak{dr}}
\end{equation}
where $H$ is as above. This can also be pulled back to a constant family of de Rham periods on $S$, which we denote by $\mathbb{L}^\mathfrak{dr}$ as well.
\end{subsubsection}
\end{subsection}

\begin{subsection}{Feynman integrals and their motivic lifts}
\label{motivic-lifts}
\begin{subsubsection}{Feynman integrals}
\begin{defi}
\label{feynman-graph-defi}
A Feynman graph is a multigraph $G$, defined by a triple:
\[
(V_G,E_G,E_G^{ext})
\]
where $V_G$ are the vertices, $E_G$ is the set of internal edges of the graph which are not oriented, with the endpoints of each element of $E_G$ encoded by a map $\partial: E_G \rightarrow \textrm{Sym}^2 V_G$, and $E_G^{ext}$ is the set of external half-edges, with the endpoint of each element of $E_G^{ext}$ encoded by a map $\partial: E_G^{ext} \rightarrow V_G$. We will only consider connected graphs. 
\end{defi}
To each internal edge $e \in E_G$ we assign its \textit{particle mass} $m_e \in \R$. To each external edge $i \in E_G^{ext}$ we assign a \textit{momentum}, which is a vector $q_i \in \R^d$ where $d \geq 0$ is the dimension of space-time. A condition on momenta $\sum_{i \in E_G^{ext}} q_i = 0$, called \textit{momentum conservation}, is assumed.  We write $G/e_i$ for the graph with the edge $e_i$ contracted, and $G/e_I$ for the graph with the edges $\{e_i\}_{i \in I}$ contracted. Denote by $M$ the number of non-zero masses of a Feynman graph and by $F$ the number of external momenta.

\begin{defi}
\label{first-and-second-symanzik}
Let $G$ be a Feynman graph. Associate to each internal edge $e \in E_G$ a variable $\alpha_e$. Then the \textit{first Symanzik polynomial} is a homogeneous polynomial defined to be
\[
\Psi_G = \sum\limits_{T\subset G} \prod\limits_{e \not \in T} \alpha_e
\]
where the sum is over all spanning trees $T$ of the graph $G$. We also define the following homogeneous polynomial
\[
\Phi_G(q) = \sum\limits_{T_1 \cup T_2 \subset G} (q^{T_1})^2 \prod\limits_{e \not \in T_1 \cup T_2} \alpha_e
\]
where the sum ranges over all spanning 2-trees $T = T_1 \cup T_2$ of $G$. A spanning 2-tree of a graph $G$ is a subgraph with 2 connected components, each of which is a tree, with the minimal number of edges such that its vertices include all of the vertices of the original graph $G$. We define $q^{T_1} = \sum_{i \in E_G^{ext}} q_i$ as the sum of all incoming momenta entering $T_1$. By momentum conservation $q^{T_1}=-q^{T_2}$. We denote the Euclidean scalar product of the vector $q \in \mathbb{R}^d$ with itself by $q^2$. The \textit{second Symanzik polynomial}, also homogeneous, is then defined to be:
\begin{equation}
\label{second-Symanzik}
\Xi_G(m,q) = \Phi_G(q) + \left(\sum\limits_{e \in E_G} m_e^2\alpha_e\right)\Psi_G.
\end{equation}
To abbreviate the dependencies of $\Xi_G$ on momenta and masses in the above definitions we write $q:=\{q_1,...,q_{F}\}$ and $m:=\{m_1,...,m_{M}\}$.
\end{defi}

Let $G$ be a Feynman graph with $N_G$ edges, $h_G$ loops. In parametric form, the Feynman integral which is of interest in physics is the following projective integral:
\begin{equation}
\label{parametric-form}
\Gamma\left(N_G - h_G\frac{d}{2}\right)\int_{\sigma} \omega_G(m,q)
\end{equation}
where
\begin{equation}
\omega_G(m,q) = \frac{1}{\Psi^{d/2}_G}\left(\frac{\Psi_G}{\Xi_G}\right)^{N_G - h_G d/2} \Omega_G
\end{equation}
and $\Psi_G$, and $\Xi_G(m,q)$ are the first and second Symanzik polynomials of $G$, respectively.

The domain of integration is:
\[
\sigma = \{ \left[\alpha_1:\ldots : \alpha_{N_G}\right] : \alpha_i \geq 0 \} \subset \Pp^{N_G - 1}(\mathbb{R})
\]
and 
\begin{equation}
\label{omega-def}
\Omega_G = \sum\limits_{i=1}^{N_G} (-1)^i\alpha_id\alpha_1 \wedge \ldots \wedge \widehat{d\alpha_i}\wedge \ldots \wedge d\alpha_{N_G}
\end{equation}
where $\widehat{d\alpha_i}$ means that we omit $d\alpha_i$. The derivation of this form of the Feynman integral from its momentum space representation using the Schwinger trick is nicely explained in \cite{PanzerPhD}.

Note that when $N_G = h_G\frac{d}{2}$ the $\Gamma$ prefactor of the previous integral will have a pole. A common regularization method in physics is to allow the number of space-time dimensions $d$ to vary, while keeping the variables $\alpha_i$ fixed, and then consider the Laurent expansion of \eqref{parametric-form} around a fixed value of $d$. For example, in this paper we want to consider integrals in $d=4$ dimensions, and dimensional regularization would amount to studying the Laurent expansion of \eqref{parametric-form} in $d=4-2\epsilon$ around the point $\epsilon=0$. If \eqref{parametric-form} diverges we will study its residue around the point $\epsilon=0$.
\begin{defi}
\label{FIntegral}
In order to consider two cases at once, we will consider the following projective integral:
\begin{equation}
\label{FIntegral-number}
I_G(m,q) = \int_{\sigma_G} \omega_G(m,q).
\end{equation}
\end{defi}
Therefore, if \eqref{parametric-form} converges we have simply dropped the prefactor which is a value of the Gamma function. If \eqref{parametric-form} does not converge then $I_G(m,q)$ is its residue in dimensional regularization.
\end{subsubsection}
\begin{subsubsection}{Motivic Feynman amplitudes}
In order to study their Galois theory we need to lift the functions $I_G(m,q)$ to families of motivic periods. Concretely, we need an element $mot_G \in \textrm{Ob}(\mathcal{H}(S))$, and a family of motivic periods
\[ \left[mot_G,[\sigma],[\omega]\right]^{\mathfrak{m}} \in \mathcal{P}^{\mathfrak{m},Z,gen}_{\mathcal{H(S)}},
\]
for a certain $S \subset K^{gen}_{F,M}$ to be defined immediately below, and $Z \subset S(\mathbb{C})$ as in \ref{motivic-periods}, such that
\[
\text{per}(\left[mot_G,[\sigma],[\omega]\right]^{\mathfrak{m}}) = I_G(m,q) \, .
\]
defines a multi-valued meromorphic function on $S(\mathbb{C})$.

The graph polynomial \eqref{second-Symanzik}, and therefore the Feynman integral, are invariant under the action of the orthogonal group in $d$ dimensions. Therefore they only depend on $s_{i,j} = s_{j,i} := q_i\cdot q_j$, the Euclidean product of $q_i,q_j$, $1 \leq i \leq j \leq F$. Recall that they also satisfy momentum conservation by assumption. Furthermore, if a graph $G$ has a vertex $v \in V$ such that the incoming momentum is non-trivial, i.e., $q^{\{v\}} \not = 0$, then the polynomial $\phi_G(q) \not = 0$ if $s_I := \sum_{i,j \in I} s_{i,j}  \not = 0$ for all $I \subsetneq \{1,\ldots,F\}$ \cite[Lemma 1.12]{Brown2}. If the previous condition holds along with $s_I + m_j^2 \not = 0$ for all $I \subsetneq \{1,\ldots,F\}$ and $j \in \{1,...,|E_G|\}$ then the  polynomial $\Xi_G(m,q)=0$ if and only if all internal masses and all external momenta are trivial \cite[Lemma 1.13]{Brown2}. This leads to the following definition. 
\begin{defi}
\label{space-of-kinematics-defi}
Define $K_{F,M} = \mathbb{A}^{\binom{F}{2}} \times \mathbb{G}^{M}_m$ to be the \textit{space of kinematics}. Let $K^{gen}_{F,M} \subset K_{F,M}$ be the open complement of the union of spaces $s_I + m^2_j = 0$, where $s_I = \sum_{i,j \in I} s_{i,j}$, for $I \subsetneq \{1,\ldots,F\}$, and $j \in \{0,1,\ldots M\}$, and we set $m_0$ = 0. We also define $U^{gen}_{F,M} \subset K^{gen}_{F,M}(\Cc)$ to be the region in $K^{gen}_{F,M}(\Cc)$ where $\Re(s_I) > 0$, and $\Re(m^2_j) > 0$, called the \textit{extended Euclidean sheet}. Denote by  $k_S = \text{Frac}(\mathcal{O}(K_{F,M}))$ the field of fractions of $\mathcal{O}(K_{F,M})$. Note that $k_S \cong \mathbb{Q}((s_{i,j})_{1 \leq i\leq j \leq F}, (m_k)_{1\leq k \leq M})$.
\end{defi}

The lifting of Feynman integrals to families of motivic periods is provided by Brown in \cite{Brown2} for Feynman graphs of any loop order. Key results for this lifting are certain factorization properties of Symanzik polynomials, which lead to the concept of \textit{motic subgraphs}. We recall the definition.
\begin{defi}
\label{motic-subgraphs}
A subgraph $\gamma$ is called \textit{mass-momentum spanning} if it contains all internal edges $e \in E_G$ which have non-vanishing mass $m_e$,  $\partial E_G^{ext} \subset V_{\gamma}$, and all the vertices $\partial E_G^{ext}$ all lie in a single connected component of $\gamma$.

A subgraph spanned by the edges $\gamma \subset E_G$, which we also denote by $\gamma$, is called \textit{motic} if, for every subgraph spanned by the edges  $\gamma^{\prime} \subsetneq E_{\gamma}$, which is mass-momentum spanning in $\gamma$, one has $h_{\gamma^{\prime}} < h_{\gamma}$, where we denote by $h_{\gamma}$ the loop number of $\gamma$. In other words, a subgraph $\gamma \subset G$ is motic if and only if removing any edge of $\gamma$ makes it non-mass-momentum spanning or reduces its loop number $h_{\gamma}$.
\end{defi}

Recall that $\alpha_e$ are projective coordinates. Let $\Delta := \bigcup_{e \in E_G}\Delta_e \subset  \Pp^{N_G-1}$, where $\Delta_e:=V(\alpha_e)  \subset \Pp^{N_G-1}$, and note that $\partial \sigma_G \subset \Delta$. Denote $\Delta_{\gamma} = \bigcap_{e \in E_{\gamma}} \Delta_e \subset  \Pp^{N_G-1}$, for a subgraph $\gamma \subset G$, and let $\Delta_{\emptyset} = \Pp^{N_G - 1}$. These schemes are defined a priory over $\Sp(\mathbb{Q})$, and we write, by abuse of notation, $\Pp^{N_G -1}$, $\Delta$, and $\Delta_{\gamma}$ for their base change to $K^{gen}_{F,M}$. 

By \cite[Proposition 6.2]{Brown2} a linear subspace $\Delta_{\gamma}$ corresponding to a motic subgraph $\gamma$ is contained in $V(\Xi_{G}(m,q))$, and if $\gamma$ is not mass-momentum spanning then $\Delta_{\gamma}$ is also contained in $V(\Psi_G)$, over each fiber over $K^{gen}_{F,M}$. Therefore, over each fiber, the boundary of the domain of integration meets the singularities of the integrand $\omega_G$, causing potential divergences. We blow up $\Pp^{N_G - 1}$ along the linear subspaces corresponding to motic subgraphs, at first blowing up those linear subspaces of dimension 0, then the strict transforms of linear subspaces of dimension one etc. Denote the blow-up by $\pi_G : P^G \rightarrow \Pp^{N_G-1}$. Let $X_G=V(\Xi_G(m,q))\bigcup V(\Psi_G), X_G^{\prime} = V(\Xi_G(m,q))$, and $Y_G, Y_G^{\prime}$ be their strict transforms. Let $D = \pi_G^{-1}(\Delta)$. 

We recall the definition of a Landau variety following \cite[Ch. IV]{Pham}, based on stratified Morse theory \cite{GM}. Consider the underlying analytic varieties of the pair $(P^G,D \bigcup Y_G)$. Then the divisor $Y = D \bigcup Y_G$ gives rise to a stratification of $P^G$
\[
P^G  \supset Y \supset Y^{(1)} \supset Y^{(2)} \supset \ldots
\]
where $Y^{(k)}$ is the skeleton of $D$ of codimension $k$. The open strata $Y^{(k)}\setminus Y^{(k+1)}$ are smooth, and the boundary of each irreducible component of $Y^{(k)}\setminus Y^{(k+1)}$, denoted $A_k$, has the property that the boundary $\bar{A_k} \setminus A_k$ is a union of strata of lower dimension. Furthermore they satisfy Whitney's conditions A and B. Define the critical set of $A_k$ to be the set where $\pi$ fails to be submersive:
\[
cA_k := \{ x \in A_k | \rk T_x\pi < \dim K^{gen}_{F,M} \} .
\]
\begin{defi}
Define the Landau variety $L_G$ to be the codimension 1 part of $\pi(\cup_{i} cA_i)$, where the union is over all the strata of $Y$.
\end{defi}

Note that we have been working with underlying analytic varieties of $P^G, Y$, and by \cite[Lemma 5.2, Corollary 5.3]{Pham} each $\pi(\cup_{i} cA_i)$ is a complex analytic set. In fact, the same proof as that of Lemma 5.2 tells us that it is in fact an algebraic variety, since it is expressed therein in terms of minors of a matrix of partial derivatives of local equations for $X_G$ and $D$, which are algebraic. Then by Thom's Isotopy theorem $\pi$ is a locally trivial fibration of stratified analytic varieties on the complement of $L_G$.

Let $S = K^{gen}_{F,M} \setminus L_G$. Note that $P^G$ is smooth and $D \cup Y_G$ is generically a simple normal crossing divisor in $P^G$. This means that we are in the situation set up in \ref{motivic-periods}. Next we write down a canonical Betti class over $Z = U^{gen}_{F,M}$, and a de Rham class over the generic point.

\begin{defi} The \textit{graph motive} of a Feynman graph $G$ is defined as
\label{graph-motive}
\[
mot_G = \left(H^{N_G-1}_B\left(P^G \setminus Y_G,D \setminus Y_G \cap D \right)_{/S},H^{N_G-1}_{\mathrm{dR}}\left(P^G \setminus Y_G, D \setminus Y_G \cap D\right)_{/S},c \right) \in \textrm{ob}(\mathcal{H}(S)).
\]
We also define
\[
mot_G^{\prime} = \left(H^{N_G-1}_B\left(P^G \setminus Y_G^{\prime},D \setminus Y_G^{\prime} \cap D \right)_{/S},H^{N_G-1}_{\mathrm{dR}}\left(P^G \setminus Y_G^{\prime}, D \setminus Y_G^{\prime} \cap D\right)_{/S},c \right) \in \textrm{ob}(\mathcal{H}(S)).
\]
\end{defi}

Let $\pi_G : P^G \rightarrow \mathbb{P}^{N_G-1}$ be the blow up along $\Delta_{\gamma}$ where these are viewed as schemes over $\Sp(\mathbb{Q})$, and let $\widetilde{\sigma}$ be the closure in the analytic topology of $\pi_G^{-1}(\accentset{\circ}{\sigma})$, where $\accentset{\circ}{\sigma}$ is the real open simplex given by $\alpha_e > 0$. We define a constant family of manifolds with corners over $U^{gen}_{F,M}$:
\[
\sigma_G = \widetilde{\sigma} \times U^{gen}_{F,M}.
\]
Then \cite[Theorem 6.7]{Brown2} tells us that
\[
\sigma_G \cap Y_G(\mathbb{C}) = \emptyset.
\]
It uses the fact that the coefficients of $\Psi_G$ are all positive and that $\Re(\Xi_G(m,q)) > 0$ when $\alpha_i > 0$ and $(m,q) \in U^{gen}_{F,M}$, i.e., parameters have positive real parts. Then we have
\[
[\sigma_G] \in \Gamma(U^{gen}_{F,M}, (mot_G)_B^{\vee}) 
\]
as desired.

Finally, for a general definition of a motivic Feynman amplitude, one needs to prove that the pull-back of the differential form $\pi_G^*(\omega_G(m,q))$ doesn't acquire any new poles along the exceptional divisors after blowing up. For a general criterion in terms of the degrees of divergence of motic subgraphs see \cite[\S 6.6]{Brown2}. We do not need this result in full generality, and will check for poles directly in section \ref{triangle}, when we first encounter motic subgraphs and having to blow-up. When this is satisfied $\pi_G^*(\omega_G(m,q))$ is a global section of $\Omega^{N_G-1}_{P^G \setminus Y_G/k_S}$ and defines a class
\[
[\pi_G^*(\omega_G(m,q))] \in \Gamma(\Sp(k_S), (mot_G)_{\mathrm{dR}})
\]
as required.
\begin{defi}
\label{motivic-Feynman-amplitude}
The \textit{motivic Feynman amplitude} is the family of motivic periods:
\[
I_G^{\mathfrak{m}}(m,q) = \left[mot_G, [\sigma_G],[\pi_G^*(\omega_G(m,q))]\right]^{\mathfrak{m}} \in \mathcal{P}^{\mathfrak{m},U^{gen}_{F,M},gen}_{\mathcal{H}(S)} 
\]
When we have $[\pi_G^*(\omega_G(m,q))] \in \Gamma(\Sp(k_S), (mot_G')_{\mathrm{dR}})$, and $[\sigma_G] \in \Gamma(U^{gen}_{F,M}, (mot_G')_B^{\vee})$ (for example when $N_G\geq (h_G+1)\frac{d}{2}$) the motivic Feynman amplitude can be considered as a motivic period of the motive $mot_G'$. In this case we denote it by the same symbol:
\[
I_G^{\mathfrak{m}}(m,q) = \left[mot_G', [\sigma_G],[\pi_G^*(\omega_G(m,q))]\right]^{\mathfrak{m}} \in \mathcal{P}^{\mathfrak{m},U^{gen}_{F,M},gen}_{\mathcal{H}(S)} .
\]
In what follows we will make it clear if we are working with $mot_G$ or $mot_G'$.

When the motive of a Feynman graph is mixed Tate (see \ref{mixed Tate HS}), as will be the case in the rest of this paper, we define the \textit{de Rham Feynman amplitude} as:
\[
I_G^{\mathfrak{dr}}(m,q) = \left[mot_G, \epsilon,[\pi_G^*(\omega_G(m,q))]\right]^{\mathfrak{dr}} \in \mathcal{P}^{\mathfrak{m},U^{gen}_{F,M},gen}_{\mathcal{H}(S)} ,
\]
where $\epsilon :  \omega^{gen}_{\rm{dR}}(mot_G)_{\rm{dR}} \rightarrow \omega^{gen}_{\rm{dR}}W_0(mot_G)_{\rm{dR}}$. This map is simply the projection to the weight 0 part. Note that the weight filtration on the de Rham realization is split since the motive is Mixed Tate.
\end{defi}

\end{subsubsection}
\end{subsection}
\end{section}

\begin{section}{Motives of one-loop graphs}
We specialize the previous discussion to the case of one-loop cycle graphs with $N_G$ internal edges, and one external edge attached to each vertex, in $d=4$ dimensions. In this case, the second Symanzik polynomial $\Xi_G$ is homogeneous quadratic in the $\alpha_i$, and the first Symanzik polynomial $\Psi_G$ is linear in the $\alpha_i$. The integral of interest is:
\begin{equation}
\label{one-loop-amplitude}
I_G(m,q) = \int\limits_{\sigma}  \frac{\Psi_G^{N_G - 4}}{\Xi_G^{N_G-2}} \Omega_G .
\end{equation}

\begin{subsection}{Four and more internal edges}
Let $N_G \geq 4$, and all masses and momenta non-vanishing, i.e., $F = M = N_G$. Note that in this case the polar locus of the integrand does not include the first Symanzik polynomial, therefore we can work with the restricted motive $mot_G^\prime$.

In the case of non-vanishing masses and momenta, there are no motic subgraphs (see definition \ref{motic-subgraphs}) of a one-loop graph. In particular there is no need to blow up, and denoting by $Q$ the vanishing locus of $\Xi_G$, and following the last section we have:
\[
mot_G' = (H^{N_G-1}_{\mathrm{dR}}(\Pp^{N_G-1} \setminus Q, \Delta \setminus Q \cap \Delta)_{/S}, H^{N_G-1}_{\mathrm{dR}}(\Pp^{N_G-1} \setminus Q, \Delta \setminus Q \cap \Delta)_{/S},c),
\]
and 
\[
[\omega_G] \in \Gamma(\Sp(k_S),(mot_G')_{\mathrm{dR}}), \quad [\sigma_G] \in \Gamma(U^{gen}_{F,M},((mot_G')_{B})^{\vee}).
\]
The associated motivic Feynman amplitude is:
\begin{defi}
\begin{equation}
\label{one-loop-motivic-feynman-ampl-defi}
I_G^{\mathfrak{m}}(m,q) = \left[mot_G', [\sigma_G],[\omega_G]\right]^{\mathfrak{m}} \in \mathcal{P}^{\mathfrak{m},U^{gen}_{F,M},gen}_{\mathcal{H}(S)} 
\end{equation}
\end{defi}

\end{subsection}

\begin{subsection}{Semi-simplification of the motive}

\begin{lemma}
\label{absolute-grps}
Let $Q \subset \Pp^n$ be a smooth quadric hypersurface. Then 
\begin{equation}
H^q(\Pp^{n} \setminus Q, \mathbb{Q}) \cong
\begin{cases}
\mathbb{Q}(-m), & \text{if } q = n = 2m-1  \\
\mathbb{Q}(0), & \text{if } q = 0 \\
0, & \text{otherwise}
\end{cases}
\end{equation}
\end{lemma}
\begin{proof}
Recall that for a quadric $Q$ of dimension $d$ we have $H^{2q}(Q, \mathbb{Q}) \cong \mathbb{Q}(-m)$ for $1\leq q \leq d$ if $d$ odd. If $d=2l$ is even then $H^{d}(Q, \mathbb{Q}) \cong \mathbb{Q}(-l) \oplus \mathbb{Q}(-l)$. Let $n=2m-1$ and consider the Gysin long exact sequence:
\[
\ldots H^{n}\left(\mathbb{P}^{n}\right) \rightarrow H^{n} \left(\mathbb{P}^{n} \setminus Q \right) \xrightarrow{\textit{res}} H^{n-1}(Q)(-1) \xrightarrow{G} H^{n+1}\left(\mathbb{P}^{n}\right) \rightarrow H^{n+1}\left(\mathbb{P}^{n} \setminus Q\right) \rightarrow \ldots
\]
Note that $H^{n-1}(Q)(-1) \cong \mathbb{Q}(-m) \oplus \mathbb{Q}(-m)$ and that the Gysin morphism $G$ is surjective since $\mathbb{P}^{n} \setminus Q$ is a closed subset of dimension $n$ of the affine space $\mathbb{P}^{N} \setminus H$, where $H$ is a hyperplane, and the closed embedding is given by the Veronese embedding $i:\mathbb{P}^n \rightarrow \mathbb{P}^{N}$, where $N = \binom{n+2}{n}-1$. Since $H^{n}\left(\mathbb{P}^{n}\right) = 0 $ we have that $H^n\left(\mathbb{P}^{n} \setminus Q \right) \cong \mathbb{Q}(-m)$. From the same long exact sequence we can see that $H^{n+k}(\mathbb{P}^n \setminus Q) = 0 $ when $0 \leq k \leq n$ because in that case $H^{n+k}(Q)(-1) \cong H^{n+k}(\mathbb{P}^n)$, and $H^{n+k}(\mathbb{P}^n) = 0$ if $k$ even and $H^{n+k-1}(Q) = 0$ if $k$ is odd. By Poincar\'e duality we get the remaining cohomology. The result is proved analogously for $n=2m$ even.
\end{proof}

Recall that the rank of a quadric $Q = V(f) \subset \mathbb{P}^n$, where $f$ is a homogeneus polynomial of degree 2, is the rank of the matrix $C$ where $f = \Vec{\alpha}C\Vec{\alpha}^t$, and $\alpha$ is the row vector $(\alpha_1,\ldots,\alpha_{n+1})$. Let $\Delta_i = V(\alpha_i)$, $\Delta_I = \cap_{i \in I} \Delta_i $, and $\Delta_{\emptyset} = \Pp^n$. The rank of $Q=V(\Xi_G(m,q)) \subset \mathbb{P}^{N_G-1}$, for $N_G \geq 6$, in $d=4$ space-time dimensions is at most 6, and it is exactly 6 for generic values of masses and momenta \cite[Lemma 6.3]{BK}, i.e., the locus where the rank is strictly less than 6 is in the Landau variety $L_G$. When $N_G \leq 5$ the quadric $Q$ is generically smooth. In order to determine the weight graded pieces of $mot_G'$, we first prove the following result, which applies to each fiber of $mot_G'$ over $S(\mathbb{C})$.

\begin{prop}
\label{weight-graded-pieces}
Let $n \geq 3$, and $\Delta \subset \Pp^n$ be the standard simplex. Let $Q \subset \Pp^n$ be a quadric of rank $min(6,n+1)$ with $Q$ and $\Delta$ in general position. Then the weight graded pieces of the  motive $ H = H^n(\Pp^n \setminus Q, \Delta \setminus Q \cap \Delta)$ are:
\begin{equation}
\begin{split}
\text{gr}^W_0 H = \mathbb{Q}(0), \quad \text{gr}^W_{2}H = \bigoplus_{\binom{n+1}{n-1}} \mathbb{Q}(-1), \quad \text{gr}^W_{4}H = \bigoplus_{\binom{n+1}{n-3}} \mathbb{Q}(-2) .
\end{split}
\end{equation}
When $n \leq 4$ all the other weight-graded pieces, except for these three, vanish. When $n \geq 5$ we also have:
\begin{equation}
\text{gr}^W_{6}H = \bigoplus_{\binom{n+1}{n-5} - \binom{n+1}{n-6}} \mathbb{Q}(-3) ,
\end{equation}
and all the others vanish.
\end{prop}
\begin{proof}
We apply the relative cohomology spectral sequence
\[
E_1^{p,q} = \bigoplus_{|I| = p} H^q(\Delta_I \setminus (Q \cap \Delta_I)) \Rightarrow H^{p+q}(\Pp^n \setminus Q, \Delta \setminus Q \cap \Delta).
\]
To compute $W_4H$ we note that our assumption on the rank of $Q$ implies that $Q \cap \Delta_I$ is smooth when $\Delta_I$ is a face  of $\Delta$ of dimension $\leq 5$. Therefore, by the previous lemma, the first row $E_1^{\bullet,1}$ is zero except for the entry $E_1^{n-1,1} \cong \bigoplus_{|I|=n-1} H^1(\Delta_I \setminus Q \cap \Delta_I ) \cong \bigoplus_{\binom{n+1}{n-1}} \mathbb{Q}(-1)$. Hence $gr^W_2 H \cong \bigoplus_{\binom{n+1}{n-1}} \mathbb{Q}(-1)$ as required. Again, by the previous lemma, row $E_1^{\bullet,2}$ is all zero, while the row $E_1^{\bullet,3}$ is zero except for $E_1^{n-3,3} \cong \bigoplus_{|I|=n-3} H^3(\Delta_I \setminus Q \cap \Delta_I ) \cong \bigoplus_{\binom{n+1}{n-3}} \mathbb{Q}(-2)$.

It remains to show the claim for the graded weight 6 piece of $H$, as well as that there are no higher weight graded pieces. By the assumption on the rank of $Q$, if $n \geq 6$ then $Q \subset \Pp^{n}$ is a generalized cone over a smooth quadric $Q_0 \subset \Pp^5$. Hence, $\Pp^{n} \setminus Q$ is a fiber bundle with fibers $\mathbb{A}^{n-5}$ over $\Pp^5 \setminus Q_0$. By the Leray spectral sequence we get  
\begin{equation}
\label{reduction1}
H^i(\Pp^{n} \setminus Q,\mathbb{Q}) \cong H^i(\Pp^5 \setminus Q_0,\mathbb{Q}) \text{ for all } i \geq 0
\end{equation}
Therefore we can apply the previous lemma again to conclude that for all $q \geq 6$ and $p \geq 0$ the $E_1^{p,q}$ vanish. Hence there are no weight graded pieces higher than 6. We also get $E_1^{n-4,5} = 0$ and 
\[
E_2^{n-5,5} \cong \text{gr}^W_{6}H \cong \coker\left(E_1^{n-6,5} \rightarrow E_1^{n-5,5}\right) \cong \bigoplus_{\binom{n+1}{n-5} - \binom{n+1}{n-6}} \mathbb{Q}(-3) .
\]
Finally, the bottom row $E_1^{\bullet,0}$ computes the cohomology of the simplicial complex $\Delta$ which is homologicaly equivalent to a sphere, giving us the graded weight 0 part of the proposition.
\end{proof}

We must be careful when pulling back this result to $S$. It is not the case that, for $S$ defined as above, $mot_G'$ is isomorphic to $\mathbb{Q}(0)_{/S} \bigoplus \mathbb{Q}(-1)^{m_1}_{/S} \bigoplus \mathbb{Q}(-2)^{m_2}_{/S} \bigoplus \mathbb{Q}(-3)^{m_3}_{/S}$, for $m_1,m_2,m_3$ defined in the previous Proposition, and $\mathbb{Q}(-k)_{/S}$, constant variation of Tate Hodge structure defined in \ref{mixed Tate HS}. The reason is that the cohomology of a smooth quadric in even dimensions has rank 2 in middle degree, and the monodromy might interchange the two classes. However, if we pass to the covering of $S$ which includes the square roots of the determinant of the matrix associated to the quadric $Q$ associated to the graph $G$, as well as the square roots of the determinants of all the quadrics $Q|_{\Delta_I}$, where $\Delta_I$ is a face of $\Delta$ as before, we do indeed get that $mot_G'$ is a constant variation. This explains the prefactors $P_{j,k}$ in Theorem 1. We assume that this has been done and denote the new base by $S$ as well.

\end{subsection}

\begin{subsection}{Reduction to four edges}

The main tool we use to reduce to the one-loop graph with 4 internal edges (Figure \ref{box-graph}) is the following:

\begin{lemma}
\label{BK-reduction}
Let $ n = N_G - 1 \geq 4$. Let $\Delta_i \subset \Pp^n$ be the face of $\Delta$ where $\alpha_i = 0$. Then the Feynman differential form $\omega_G(m,q)$ is exact, and we can find an $(n-1)$-form $\omega_{n-1}$ which is a global section of the sheaf $\Omega^{n-1}_{\Pp^n\setminus Q/k_S}$, along with constants $a_j \in k_S$, such that $d\omega_{n-1} = \omega_G$ and $\omega_{n-1}|_{\Delta_j} = a_j \omega_{G/e_j}$.
 \end{lemma}
\begin{proof}
See \cite[Lemma 8.1]{BK}
\end{proof}

\begin{prop} \label{reduction-lemma}
For generic values of masses and momenta and in $d=4$ dimensions, the motivic Feynman amplitude of any one-loop graph $G$ with $N_G \geq 5$ edges is a $k_S$-linear combination of motivic periods of graph motives of 4-edge quotient graphs of $G$.
\end{prop}
\begin{proof}

We start with the motivic period $\left[mot_G',[\sigma_G],[\omega_G]\right]^{\mathfrak{m}}$, and we would like to show that there exist constants $a_I \in k_S$, for $I \subset \{1,\ldots,N_G\}$, such that we can write
\begin{equation}
\left[mot_G', [\sigma_G], [\omega_G]\right]^{\mathfrak{m}} = \sum_{\vert I \vert = N_G-4} a_I \cdot \left[mot_{G/e_I}', \left[\sigma_{G/e_I}\right], \left[\omega_{G/e_I}\right]\right]^{\mathfrak{m}}
\end{equation}
or equivalently:
\begin{equation}
\label{decomp-reduction}
I^{\mathfrak{m}}_G = \sum_{\vert I \vert = N_G-4} a_I \cdot I^{\mathfrak{m}}_{G/e_I}.
\end{equation}

Restricting the family $\Pp^{N_G-1}\setminus Q$, and $\Delta \setminus Q \cap \Delta$ to the fiber at the generic point, we apply the previous lemma as follows. Because $\Pp^{N_G-1}\setminus Q$ is affine and $\Delta \setminus Q \cap \Delta$ is a simple normal crossing divisor therein, $(mot'_G)_{\mathrm{dR}}$ is computed by the cohomology of the total complex of the double complex 
\begin{equation}
\label{dbl-complex}
\Gamma(\Pp^n \setminus Q,\Omega^{\bullet}_{\Pp^n \setminus Q}) \rightarrow \bigoplus\limits_{|I| = 1} \Gamma(\Delta_I \setminus \Delta_I \cap Q, \Omega_{\Delta_I \setminus \Delta_I \cap Q}^{\bullet}) \rightarrow \ldots \rightarrow \bigoplus\limits_{|I| = n} \Gamma(\Delta_I \setminus \Delta_I \cap Q, \Omega_{\Delta_I \setminus \Delta_I \cap Q}^{\bullet})
\end{equation}
where $\Omega_{\Delta_I}^{\bullet}$ is the direct image under $\Delta_I \hookrightarrow \Pp^n \setminus Q$ of the sheaf of algebraic differential forms on $\Delta_I$ and vanishes outside of $\Delta_I$. The horizontal morphisms are restrictions, i.e. pullback along the inclusion of faces $\Delta_{I} \hookrightarrow \Delta_{I \setminus i_j}$ with the sign $(-1)^k$ where $i_j$ is the $k$th element of $I$. The differential of the total complex of (\ref{dbl-complex}) is defined by:
\begin{equation}
\label{diff-of-total-complex}
d_{Tot}^n = \sum_{n=p+q} d^{p,q} + (-1)^p d_2^{p,q}
\end{equation}
where $d$ is the exterior derivative and $d_2$ is the restriction. Recall that we have specialized to the generic point, which we drop from the notation. Therefore the coefficients of the algebraic differential forms above lie in the field $k_S$. This double complex is obtained by specializing (\ref{relative-double-complex}) to a point and using the fact that we are working with affine schemes.

Therefore, the class of the Feynman form $[\omega_G] \in (mot'_G)_{\mathrm{dR}}$ can be represented by an element: $(\omega_G,0 \ldots, 0) \in \bigoplus_{N_G-1 = p+q} C^{p,q}$, where $C^{p,q} = \bigoplus_{\vert I \vert = p} \Gamma\left(\Delta_I \setminus Q \cap \Delta_I,\Omega^q_{\Delta_I \setminus Q \cap \Delta_I}\right)$.

By Lemma \ref{BK-reduction} we can construct an $(N_G-2)$-form 
\[\omega':=(\omega, 0, \ldots,0) \in \bigoplus_{N_G-2 = p+q} C^{p,q}, \quad \text{where }\omega \in \Gamma\left(\Pp^{N_G-1} \setminus Q,  \Omega^{N_G-2}_{\Pp^{N_G-1} \setminus Q}\right), \]
and an $(N_G-1)$-form 
\begin{equation}
\label{decomposition}
\begin{split}
\omega'' & := \left(0, \omega|_{\Delta_1}, \omega|_{\Delta_2},\ldots, \omega|_{\Delta_{N_G}},0,\ldots,0\right) = \\ 
& = \left(0, a_1 \cdot \omega_{G/e_1}, a_2 \cdot \omega_{G/e_2},\ldots,a_{N_G} \cdot \omega_{G/e_{N_G}},0,\ldots,0\right)  \in \bigoplus_{N_G-1 = p+q} C^{p,q}
\end{split}
\end{equation}
such that $d_{Tot} (\omega') = \omega + \omega''$. Now, obviously we can write $[\omega'']$ as a linear combination, over $k_S$, of classes of $(N_G-2)$-forms in $(mot_{G/e_j})_{\mathrm{dR}}$ for $1 \leq j \leq N_G$.

The de Rham component of a face map (Definition \ref{face-map}):
\begin{equation}
\label{face-map-in-use}
\Phi_{i,dR} :  \Gamma(\Sp(k_S),(mot_{G/e_i}')_{\mathrm{dR}}) \rightarrow \Gamma(\Sp(k_S),(mot_G')_{\mathrm{dR}}),
\end{equation}
for all $1\leq i \leq N_G$ sends
\[
\left[\left(a_i \cdot \omega_{G/e_i},0,\ldots,0\right)\right] \mapsto \left[\left(0,\ldots,0,a_i \cdot \omega_{G/e_i},0,\ldots,0
\right)\right]
\]
Summing over all the faces we get a morphism
\begin{equation}
\begin{split}
\Phi_{\mathrm{dR}} : & \bigoplus_i \Gamma(\Sp(k_S),(mot_{G/e_i}')_{\mathrm{dR}}) \rightarrow \Gamma(\Sp(k_S),(mot_G')_{\mathrm{dR}}) \\
& \left( a_1 \cdot \left[\omega_{G/e_1}\right], \ldots, a_{N_G} \cdot \left[\omega_{G/e_{N_G}}\right]\right) \mapsto [\omega_G]
\end{split}
\end{equation}
where we write $[\omega_{G/e_i}]$ for the class of the element $(\omega_{G/e_i},0,\ldots,0)$. 

The Betti components of face maps (\ref{face-map}) are given by restriction to the faces of $\Delta \setminus \Delta \cap Q$, and again we sum over all the faces to get:
\begin{equation}
\begin{split}
\Phi_B^{\vee}: & \Gamma(U^{gen}_{F,M},(mot_G')^{\vee}_{\mathrm{B}}) \rightarrow \bigoplus_i \Gamma(U^{gen}_{F,M},(mot_{G/e_i}')^{\vee}_{\mathrm{B}})\\ 
& [\sigma_G] \mapsto\left( \left[\sigma_{G/e_1}\right], \ldots,  \left[\sigma_{G/e_{N_G}}\right]\right)
\end{split}
\end{equation}
which shows the equivalence of motivic periods:
\[
[mot_G, [\sigma_G], [\omega_G]]^{\mathfrak{m}} = \sum_{1 \leq i \leq N_G} a_i \cdot \left[mot_{G/e_i}, \left[\sigma_{G/e_i}\right], \left[\omega_{G/e_i}\right]\right]^{\mathfrak{m}}
\]
Iterating this same process we can then write each of the $(N_G-2)$-forms on the right hand side above as a $k_S$-linear combination of $(N_G-3)$-forms of graphs obtained by contracting two edges of $G$, up to a form that is exact in relative de Rham cohomology. This can be repeated until we get to a $k_S$-linear combination of $3$-forms which belong to the de Rham realizations of motives of graphs with 4 edges obtained by contracting the $N_G-4$ edges of the original graph. Gathering all the coefficients $a_i$ at each stage into $a_I$ for each set of $N_G-4$ contracted edges $I$, we obtain the proof of the lemma.
\end{proof}

Applying the period map to (\ref{decomp-reduction}) we recover an old result of Nickel \cite{N}:

\begin{cor}
For generic values of masses and momenta the Feynman integral of any one-loop graph $G$ in $d=4$ dimensions is a $k_S$-linear combination of periods of graph motives of 4-edge quotient graphs of $G$.
\end{cor}

\begin{rmk}
Note that using face maps to show equivalences of families of motivic periods in the previous proposition is the "motivic lift" of repeatedly applying Stokes' theorem on the corresponding integrals under the period map.
\end{rmk}

\begin{rmk}
\label{they-are-dilogs-rmk}
Since we have reduced to motives of graphs with 4 edges in the previous proposition, proposition \ref{weight-graded-pieces} implies that all motivic Feynman amplitudes of one-loop graphs with generic masses and momenta are $k_S$-linear combinations of families of periods of motives that have the following weight graded pieces
\[
\text{gr}^W (mot_G) = \mathbb{Q}(-2)_{/_S} \oplus \mathbb{Q}(-1)_{/_S}^{\oplus 6} \oplus \mathbb{Q}(0)_{/_S}
\]
Griffiths transversality is used in \cite[§9]{BK} to show that, up to a constant of integration, periods of heighest weight of such motives are always $k_S$-linear combinations of dilogarithms.
\end{rmk}
\end{subsection}

\begin{subsection}{Fewer than four edges}
\begin{subsubsection}{The triangle graph}
We will also consider the graph with $N_G=3$ edges, henceforth \textit{the triangle graph}. In the next figure we have the triangle graph when $F=M=N_G$, i.e., all masses and momenta are non-zero. Recall that denote the internal edges with non-vanishing masses with double lines.

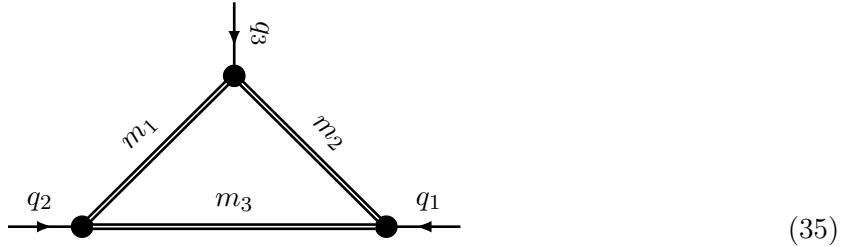
\begin{equation}
\label{general-triangle}
\begin{tikzpicture}
\SetGraphUnit{1}
  
  \SetUpEdge[lw = 1pt,
  color      = black,
  labelcolor = white,
  labelstyle = {sloped,above,yshift=2pt}]
  
  \SetUpVertex[FillColor=black, MinSize=8pt, NoLabel]

  \Vertex[x=4,y=2]{0}
  \Vertex[x=2,y=0]{1}
  \Vertex[x=6,y=0]{2}
  \Vertex[x=4,y=3,empty=true]{3}
  \Vertex[x=1,y=0,empty=true]{4}
  \Vertex[x=7,y=0,empty=true]{5}

  \Edge[label=$m_1$,style = {double}](0)(1)
  \Edge[label=$m_2$,style = {double}](0)(2)
  \Edge[label=$m_3$,style = {double}](2)(1)
  \tikzset{EdgeStyle/.style={postaction=decorate,decoration={markings,mark=at position 0.7 with {\arrow{latex}}}}}
  \Edge[label=$q_3$](3)(0)
  \Edge[label=$q_2$](4)(1)
  \Edge[label=$q_1$](5)(2)
  \tikzset{EdgeStyle/.append style = {bend left}}
\end{tikzpicture}
\end{equation}
The first and second Symanzik polynomials of this graph are:
\begin{equation}
\begin{split}
&\Psi_G = \alpha_1 + \alpha_2 + \alpha_3, \text{ and} \\
&\Xi_G(m,q) =  q_1^2\alpha_2\alpha_3 + q_2^2\alpha_1\alpha_3+q_3^2\alpha_1\alpha_2 + (m_1^2\alpha_1+m_2^2\alpha_2+m_3^2\alpha_3)\Psi_G .
\end{split}
\end{equation}

Note that the differential form corresponding to this graph $\omega_G$ (\ref{one-loop-amplitude}) has the first Symanzik polynomial in the denominator. Because of this we must consider the full graph motive. Let $L = V(\Phi_G)$, and $Q=V(\Xi_G(m,q))$. Then, the motivic Feynman amplitude is a family of periods associated to the object:
\[
mot_G := H^2(\Pp^2 \setminus (Q \cup L), \Delta \setminus (Q \cup L) \cap \Delta)_{/S}
\]
which we refer to as \textit{the motive of the triangle graph}. We will study the triangle graph and the Galois coaction on the associated motivic Feynman amplitude in section \ref{triangle}. We will also look at cases of the triangle graph where internal masses vanish.
\end{subsubsection}

\begin{subsubsection}{The bubble graph}
Finally, we will be considering the one loop graph with $N_G=2$ edges, henceforth \textit{the bubble graph}. The first case we will look at is $F=M=N_G$, when the masses and the momentum are non-vanishing. 
\begin{equation}
\label{general-bubble}
\begin{tikzpicture}

\SetGraphUnit{3}
  
  \SetUpEdge[lw = 1pt,
  color      = black,
  labelcolor = white,
  labelstyle = {sloped,above,yshift=2pt}]
  
  \SetUpVertex[FillColor=black, MinSize=8pt, NoLabel]

  \Vertex[x=2,y=0]{1}
  \Vertex[x=6,y=0]{2}
  \Vertex[x=0,y=0,empty=true]{4}
  \Vertex[x=8,y=0,empty=true]{5}

  \tikzset{EdgeStyle/.append style = {bend left}}
  \Edge[label=1,style = {double}](1)(2)
  \Edge[label=2,style = {double}](2)(1)
  \tikzset{EdgeStyle/.style={postaction=decorate,decoration={markings,mark=at position 0.7 with {\arrow{latex}}}}}
  \Edge[label=$-q_1$](4)(1)
  \Edge[label=$q_1$](5)(2)
\end{tikzpicture}
\end{equation}
The first and second Symanzik polynomials of this graph are:
\begin{equation}
\begin{split}
\Psi_G &= \alpha_1 + \alpha_2, \text{ and} \\
\Xi_G(m,q) &= q_1^2\alpha_1\alpha_2 + (m_1^2\alpha_1+m_2^2\alpha_2)\Psi_G .
\end{split}
\end{equation}
The motive of the bubble graph is 
\begin{equation}
\label{bubble-motive}
mot_G := H^1(\Pp^1 \setminus (Q \cup L), \Delta \setminus (Q \cup L) \cap \Delta)_{/S},
\end{equation}
where $L = V(\Phi_G)$, and $Q=V(\Xi_G(m,q))$.
It can easily be seen from the relative cohomology long exact sequence 
\[
\ldots \rightarrow H^0(\Delta \setminus (Q \cup L) \cap \Delta)_{/S} \rightarrow H^1(\Pp^1 \setminus (Q \cup L), \Delta \setminus (Q \cup L) \cap \Delta)_{/S} \rightarrow H^1(\Pp^1 \setminus (Q \cup L))_{/S} \rightarrow \ldots  
\]
that it has rank 3.

The Feynman form in $d=4$ dimensions is 
\begin{equation}
\label{bubble-Feynman-form}
\omega_G = \frac{\alpha_2d\alpha_1 - \alpha_1d\alpha_2}{\Psi_G^2}
\end{equation}
and the associated Feynman integral is $I_G = 1$. In order to relate the bubble graph with the motivic coaction on one loop graphs with more than 2 edges, we are going to consider the motivic Feynman amplitude of this graph in $d=2$ dimensions, in which case we have the Feynman form
\[
\theta^1_G = \frac{\alpha_2d\alpha_1 - \alpha_1d\alpha_2}{\Xi_G}.
\]
It defines a class over the generic point of the de Rham realization of the restricted motive of the bubble graph:
\[
mot_G' := H^1(\Pp^1 \setminus Q, \Delta \setminus Q \cap \Delta)_{/S}
\]
Denote by 
\[
[p_0p_1|p_2p_3] = \frac{(p_2-p_0)(p_3-p_1)}{(p_2-p_1)(p_3-p_0)}
\]
the cross-ratio of 4 points on $\Pp^1$. Then the period corresponding to the motivic Feynman amplitude of the bubble graph in $d=2$ dimensions is:
\begin{equation}
\begin{split}
\mathrm{per}\left(\left[mot_G,[\sigma_G],\left[\theta^1_G\right]\right]^{\mathfrak{m}}\right) & = \frac{1}{\sqrt{4|\det{C}|}} \mathrm{log}([p_0p_1|u_0u_1]) = \\ & = \frac{1}{x - y} \mathrm{log}\left(\frac{y}{x}\right),
\end{split}
\end{equation}
where $\{u_0,u_1\} := V(\Xi_G) \subset \Pp^1$, $C$ is the matrix of the quadratic form $\Xi_G$, $x,y$ are the coordinates of $u_0,u_1$ respectively in the chart where $\alpha_2 \not = 0$,  and $p_0=[0:1],p_1=[1:0]$. We know from cohomology computations that $\left[mot_G,[\sigma_G],\left[\theta^1_G\right]\right]^{\mathfrak{m}}$ is a motivic logarithm (up to an explicit prefactor), and since these are determined by their periods we have that $\left[mot_G,[\sigma_G],\left[\theta^1_G\right]\right]^{\mathfrak{m}} = \frac{1}{\sqrt{4|\det{C}|}} \log^{\mathfrak{m}}([p_0p_1|u_0u_1])$.
We will also consider the bubble graph with one vanishing mass:
\begin{equation}
\begin{tikzpicture}

\SetGraphUnit{3}
  
  \SetUpEdge[lw = 1pt,
  color      = black,
  labelcolor = white,
  labelstyle = {sloped,above,yshift=2pt}]
  
  \SetUpVertex[FillColor=black, MinSize=8pt, NoLabel]

  \Vertex[x=2,y=0]{1}
  \Vertex[x=6,y=0]{2}
  \Vertex[x=0,y=0,empty=true]{4}
  \Vertex[x=8,y=0,empty=true]{5}

  \tikzset{EdgeStyle/.append style = {bend left}}
  \Edge[label=1](1)(2)
  \Edge[label=2,style = {double}](2)(1)
  \tikzset{EdgeStyle/.style={postaction=decorate,decoration={markings,mark=at position 0.7 with {\arrow{latex}}}}}
  \Edge[label=$-q$](4)(1)
  \Edge[label=$q$](5)(2)
\end{tikzpicture}
\end{equation}
and we can check that one of the two points of $Q = \{u_0,u_1\} \subset \Pp^1$ coincides with one of the points of $\Delta = \{[1:0],[0:1] \}$. Without loss of generality let $u_0 \in \Delta$. In that case we shall consider the restricted motive of the bubble graph with one vanishing mass:
\begin{equation}
\label{bubble-motive-1mass0}
mot_G'' = H^1(\Pp^1 \setminus (u_1 \cup L), \Delta \setminus (u_1 \cup L) \cap \Delta)_{/S}
\end{equation}
It has rank 2, by the same argument as before. Note that in $d=4$ dimensions the Feynman integral is still 1, and moreover, in $d=2$ dimensions, the Feynman integral diverges. There is still one interesting period of this motive which we get by pairing the interval $(0,\infty)$ with a generator of $(mot_G'')_{\mathrm{dR}}$ represented by the form
\begin{equation}
\label{diff-form-2}
\theta^2_G = \frac{(x-y)(\alpha_2d\alpha_1 - \alpha_1d\alpha_2)}{(\alpha_2-x\alpha_1)(\alpha_2-y\alpha_1)}
\end{equation}
 where $x,y$ are coordinates of $u_1$ and $L$ respectively in the coordinate chart of $\Pp^1$ where $\alpha_2 \not = 0$. The associated period is:
\[
\mathrm{per}\left(\left[mot_G'',[\sigma_G],\left[\theta^2_G\right]\right]^{\mathfrak{m}}\right) = \frac{1}{x - y} \mathrm{log}\left(\frac{y}{x}\right) .
\]
Similarly to the perious case we see that $\left[mot_G'',[\sigma_G],\left[\theta^2_G\right]\right]^{\mathfrak{m}}$ is a motivic logarithm.
\end{subsubsection}
\end{subsection}
\end{section}

\begin{section}{Coaction on the 4-edge graph: proof of Theorem 1.}
\begin{subsection}{The motivic side of the coaction}
\label{coaction-4-edge-computation}
The first and second Symanzik polynomials of the one-loop graph with 4 edges $G$ are:
\begin{equation}
\label{Symanzik-four-edge}
\begin{split}
\Psi_G & = (\alpha_1 + \alpha_2 + \alpha_3 + \alpha_4) \\
\Xi_G &= (q_2+q_3)^2\alpha_1\alpha_3 + (q_1+q_2)^2\alpha_2\alpha_4 + q_1^2\alpha_1\alpha_4 + q_2^2\alpha_1\alpha_2 + q_3^2\alpha_2\alpha_3 + q_4^2\alpha_3\alpha_4 \\
&+ (m_1^2\alpha_1 + m_2^2\alpha_2 + m_3^2\alpha_3 + m_4^2\alpha_4) \Psi_G .
\end{split}
\end{equation}
Recall that the general formula for the coaction is:
\[
\Delta \left[mot_G', [\sigma_G], [\omega_G]\right]^{\mathfrak{m}} = \sum\limits_{j} \left[mot_G', [\sigma_G], [\omega_j]\right]^{\mathfrak{m}} \otimes  \left[mot_G', [\omega_j]^{\vee}, [\omega_G]\right]^{\mathfrak{dr}}
\]
where $\{[\omega_j]\}$ is a basis of $(mot_G')_{\mathrm{dR}}$, and $\{[\omega_j]^{\vee}\}$ is the dual basis. The first sheet of the relative cohomology spectral sequence
\[
E_1^{p,q} = \bigoplus_{|I| = p} H^q(\Delta_I \setminus (Q \cap \Delta_I)) \Rightarrow H^{p+q}(\Pp^n \setminus Q, \Delta \setminus Q \cap \Delta)
\]
applied to any fiber over $t \in S(\mathbb{C})$ of $(mot_G')_B$, by lemma \ref{absolute-grps}, reads:
\begin{equation}
\label{4-edge-ss}
\begin{tikzcd}
E^{0,3} \cong \mathbb{Q}(-2) \arrow{r} & E^{1,3} \cong 0  \arrow{r} & E^{2,3} \cong 0 \arrow{r} & E^{3,3} \cong 0 \\
E^{0,2} \cong 0 \arrow{r} & E^{1,2} \cong 0  \arrow{r} & E^{2,2} \cong 0 \arrow{r} & E^{3,2} \cong 0 \\
E^{0,1} \cong 0 \arrow{r} & E^{1,1} \cong 0  \arrow{r} & E^{2,1} \cong \mathbb{Q}(-1)^{\oplus 6} \arrow{r} & E^{3,1} \cong 0 \\
E^{0,0} \cong \mathbb{Q}(0) \arrow{r} & E^{1,0} \cong  \mathbb{Q}(0)^{\oplus 4} \arrow{r} & E^{2,0} \cong  \mathbb{Q}(0)^{\oplus 6} \arrow{r} & E^{3,0} \cong \mathbb{Q}(0)^{\oplus 4} .\\
\end{tikzcd}
\end{equation}
We can see that the morphism 
\[
\bigoplus_{|I| = 2} H^1(\Delta_I \setminus Q \cap \Delta_I, \Delta^I)_{/S} \rightarrow W_2 H^3(\Pp^3 \setminus Q, \Delta \setminus Q \cap \Delta) ,
\]
given by the sum of face maps, is surjective by computing the relative cohomology spectral sequence for a fiber of $H^1(\Delta_I \setminus Q \cap \Delta_I, \Delta^I)_{/S}$, where $|I|=2$, and noticing that the morphism of spectral sequneces induced by the sum of face maps is surjective. Now we look at the de Rham and Betti realization of this sum of face maps.

Consider the bubble graph obtained by contracting edges $e_i$ and $e_j$. In the next figure we have chosen to contract $e_2$ and $e_3$:
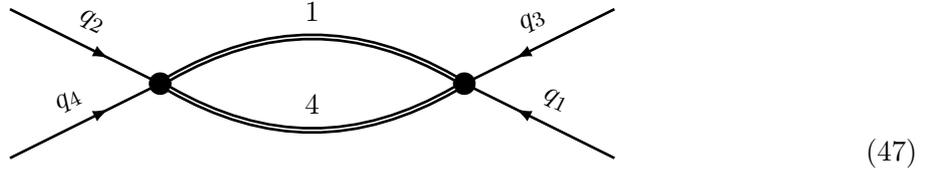
\begin{equation}
\begin{tikzpicture}

   \SetGraphUnit{1}
  \SetUpEdge[lw = 1pt,
  color      = black,
  labelcolor = white,
  labelstyle = {sloped,above,yshift=2pt}]
  
  \SetUpVertex[FillColor=black, MinSize=8pt, NoLabel]

  \Vertex[x=2,y=0]{1}
  \Vertex[x=6,y=0]{2}
  \Vertex[x=0,y=1,empty=true]{4}
  \Vertex[x=8,y=1,empty=true]{5}
  \Vertex[x=0,y=-1,empty=true]{6}
  \Vertex[x=8,y=-1,empty=true]{7}

  \tikzset{EdgeStyle/.append style = {bend left}}
  \Edge[label=1,style = {double}](1)(2)
  \Edge[label=4,style = {double}](2)(1)
  \tikzset{EdgeStyle/.style={postaction=decorate,decoration={markings,mark=at position 0.7 with {\arrow{latex}}}}}
  \Edge[label=$q_2$](4)(1)
  \Edge[label=$q_3$](5)(2)
  \Edge[label=$q_4$](6)(1)
  \Edge[label=$q_1$](7)(2)
\end{tikzpicture}
\end{equation}
The restricted motive of the bubble graph (\ref{bubble-motive}) obtained by contracting edges $e_i$ and $e_j$ is:
\[
mot_{G/\{e_i,e_j\}}' = H^1\left(\Pp^1 \setminus Q_{G/\{e_i,e_j\}}, \Delta \setminus Q_{G/\{e_i,e_j\}} \cap \Delta \right)_{/S}
\]
where $Q_{G/\{e_i,e_j\}}$ is the vanishing locus of the second Symanzik polynomial of the bubble graph. Observe that $\Xi_{G/\{e_i,e_j\}} = \Xi_G |_{\{\alpha_i = 0, \alpha_j = 0\}}$, hence:
\[
 H^1(\Delta_I \setminus Q \cap \Delta_I, \Delta^I)_{/S} \cong mot_{G/\{e_i,e_j\}}' ,
\]
where $I = \{i,j\}$. We can then choose a basis of the fiber over the generic point of $(mot_G)_{\mathrm{dR}}$ such that it contains the class of the Feynman form $[\omega_G]$ of $G$ , which is of weight 4, and the 6 weight 2 classes denoted by $\left[\omega^{ij}_G\right]$, which we define to be the images of $\left[\theta^1_{G/\{e_i,e_j\}}\right]$ under the de Rham component of the face map (see Definition \ref{face-map}):
\begin{equation}
\begin{split}
\Phi_{ij,dR} : &\Gamma(\Sp(k_S),(mot'_{G/\{e_i,e_j\}})_{\mathrm{dR}}) \rightarrow \Gamma(\Sp(k_S),(mot'_G)_{\mathrm{dR}}) \\
& \left[\theta^1_{G/\{e_i,e_j\}}\right] \mapsto \left[\omega^{ij}_G\right]
\end{split}
\end{equation}
for all $1 \leq i < j \leq 4$.
The restriction to the faces of the domain of integration 
\[
\Phi_B^{\vee}: [\sigma_G] \mapsto [\sigma_{G/\{e_i,e_j\}}]
\]
is the Betti component of the face map, which gives an equivalence of motivic periods:
\[
\left[mot_G', [\sigma_G], \left[\omega^{ij}_G\right] \right]^{\mathfrak{m}} = \left[mot'_{G/\{e_i,e_j\}},  [\sigma_{G/\{e_i,e_j\}}], \left[\theta^1_{G/\{e_i,e_j\}}\right]\right]^{\mathfrak{m}} = I_{G/\{e_i, e_j\}}^{\mathfrak{m}}\left(\theta^1_{G/\{e_i,e_j\}}\right), 
\]
where $I_{G/\{e_i, e_j\}}^{\mathfrak{m}}\left(\theta^1_{G/\{e_i,e_j\}}\right)$ is the motivic Feynman amplitude of the bubble graph in $d=2$ dimensions.
We have proved:
\begin{prop}
Let $G$ be a one loop graph with $N_G = 4$. Then the motivic Galois coaction on its motivic Feynman amplitude in $d=4$ dimensions with non-vanishing masses and momenta is:
\begin{equation}
\label{coaction-4-edge}
\begin{split}
\Delta I_G^{\mathfrak{m}} & = I_G^{\mathfrak{m}} \otimes (\mathbb{L}^{\mathfrak{dr}})^2 + \sum_{1 \leq i<j \leq 4} I_{G/\{e_i,e_ j\}}^{\mathfrak{m}}\left(\theta^1_{G/\{e_i,e_ j\}}\right) \otimes  \left[mot'_{G}, \left[\omega^{ij}_{G}\right]^{\vee},[\omega_{G}]\right]^{\mathfrak{dr}}  + 1 \otimes I^{\mathfrak{dr}}_G .
\end{split}
\end{equation}
\end{prop}

\begin{cor}
We can rewrite the coaction in terms of motivic logarithms as
\[
\Delta I_G^{\mathfrak{m}} = I_G^{\mathfrak{m}} \otimes (\mathbb{L}^{\mathfrak{dr}})^2 + \sum_{1 \leq i<j \leq 4}\frac{1}{\sqrt{4 |\det D_{i,j}|}} \log^{\mathfrak{m}}\left(\left[p_0p_1 | u_{\{i, j\}}^0u_{\{i, j\}}^1\right]\right) \otimes  \left[mot'_{G}, \left[\omega^{ij}_{G}\right]^{\vee}, [\omega_{G}]\right]^{\mathfrak{dr}} + 1 \otimes I^{\mathfrak{dr}}_G ,
\]
where $D_{i,j}$ is the matrix of the quadratic form $\Xi_{G/\{e_i,e_ j\}}$.
\end{cor}
\begin{proof}
We have
\[
mot'_{G/\{e_i, e_j\}}\cong  H^1(\Pp^1 \setminus \{ u^0_{\{i, j\}}, u^1_{\{i, j\}}, \{ p_0, p_1]\})_{/S}
\]
where 
\begin{equation}
\begin{split}
& \{u_{\{i, j\}}^0,u_{\{i, j\}}^1\} = Q_{G/\{e_i, e_j\}} \cap {\Delta_{\{i,j\}}} \\
& \{p_0,p_1\} = \{[1:0],[0:1]\}
\end{split}
\end{equation}
and the form $\theta^1_{G/\{e_i,e_ j\}}$ has simple poles at $\{u_{ \{i, j\}}^0,u_{\{i, j\}}^1\}$. The result follows.
\end{proof}

\begin{rmk}
One can combine the previous proposition with Proposition \ref{reduction-lemma} to obtain the coaction on all one-loop graphs with $N_G \geq 4$ edges in $d=4$ dimensions with non-vanishing masses and momenta.
\end{rmk}
\end{subsection}

\begin{subsection}{de Rham side of the coaction}
\label{de-rham-side}
In order to complete the proof of Theorem 1 we need to determine the de Rham side of the coaction, i.e. to interpret the objects $\left[mot'_{G}, \left[\omega^{ij}_{G}\right]^{\vee}, [\omega_{G}]\right]^{\mathfrak{dr}}$. 

\begin{subsubsection}{Connection and the de Rham periods}
\label{connection-and-de-rham-periods}

The connection allows us to differentiate a section $[\omega]$ of  $H_{\mathrm{dR}}^n(X,D)_{/S}$ with respect to parameters, which in the case of Feynman amplitudes are the masses and momenta. Consider a section $[\omega]$ of $\mathcal{V}_{\mathrm{dR}}$ on an open affine $U \subset S$ with coordinate $q$. We can compose the connection with the contraction by a vector field $\frac{\partial}{\partial q}$ to obtain a map:
\[
\mathcal{V}_{\mathrm{dR}} \xrightarrow{\nabla} \mathcal{V}_{\mathrm{dR}} \otimes \Omega^1_{S/k} \xrightarrow{\frac{\partial}{\partial q}} \mathcal{V}_{\mathrm{dR}}
\]
which sends $[\omega]$ to its first derivative with respect to $q$, denoted $\nabla_q (\omega)$. Since $\mathcal{V}_{\mathrm{dR}}$ has finite rank we will get a relation between
\[
[\omega], \nabla_q([\omega]), \nabla_q^2 ([\omega]),...,\nabla_q^k ([\omega])
\]
for some finite $k$. This is the Picard-Fuchs equation satisfied by $\omega$ and therefore by per$([\mathcal{V}, [\sigma], [\omega]]^{\mathfrak{m}})$, for some cycle of integration $\sigma \in (\omega^X_B(\mathcal{V}))^{\vee}$, defined over some simply connected $X \subset S(\mathbb{C})$. 

We use the Gauss--Manin connection on the vector bundle
\[
(mot'_G)_{\mathrm{dR}} = H_{\mathrm{dR}}^3(\Pp^3 \setminus Q, \Delta \setminus Q \cap \Delta)_{/S}
\]
 This vector bundle sits in a long exact sequence of vector bundles:
\[
\ldots \rightarrow H_{\mathrm{dR}}^2(\Delta \setminus Q \cap \Delta)_{/S} \rightarrow H_{\mathrm{dR}}^3(\Pp^3 \setminus Q, \Delta \setminus Q \cap \Delta)_{/S} \rightarrow H_{\mathrm{dR}}^3(\Pp^3 \setminus Q)_{/S} \rightarrow \ldots , 
\]
where $H^k_{\rm{dR}}(\Delta \setminus Q \cap \Delta)_{/S}$ are objects of $\mathcal{H}(S)$ which are obtained by truncating the complexes \eqref{relative-double-complex-Betti} and \eqref{relative-double-complex} on the left so that the non-zero components are in $|J| \geq n-1$,and we denote the section of $H_{\mathrm{dR}}^3(\Pp^3 \setminus Q)_{/S}$ over the generic point which is the image of $[\omega_G]$ by the same symbol. Since $H_{\mathrm{dR}}^3(\Pp^3 \setminus Q)_{/S}$ is a vector bundle with connection of rank 1 we know that the Feynman integrand satisfies a relation $\nabla_{q_1}([\omega_G]) + B(m,q)[\omega_G] = 0$, where $\nabla_{q_1}$ is the Gauss--Manin connection on $H_{\mathrm{dR}}^3(\Pp^3 \setminus Q)_{/S}$ composed with contraction by the vector field $\partial/\partial q_1$, and $B(m,q) \in k_S$. This relation lifts to a relation of sections of $H_{\mathrm{dR}}^3(\Pp^3 \setminus Q, \Delta \setminus Q \cap \Delta)_{/S}$. We compute it explicitly in the following lemma:

\begin{lemma}
\label{computing-picard-fuchs}
Let $q_1$ be one of the momentum parameters of a one-loop four-edge integral $I_G(m,q)$ with non-vanishing masses and momenta. Then $[\omega_G] \in \Gamma (\Sp(k_S),(mot_G')_{\mathrm{dR}})$ satisfies a relation of Picard-Fuchs type
\begin{equation}
\label{Picard-Fuchs}
\nabla_{q_1}([\omega_G]) + B(m,q)[\omega_G] = [d\beta] 
\end{equation}
where $\beta$ is a section of $\Omega^2_{\Pp^3 \setminus Q/\Sp(k_S)}$, and $B(m,q) \in k_S$. Both $\beta$ and $B(m,q)$ can be computed explicitly.
\end{lemma}

\begin{proof}
We use the general description of rational forms on $\Pp^n$ with poles along a hypersurface \cite{G}.  A 2-form on $\Pp^3$ with a pole along $\Xi_G$ of order 2 is of the form:
\begin{equation}
\label{beta}
\beta = \frac{\sum\limits_{i<j} (-1)^{i+j}\alpha_jA_i - \alpha_iA_j}{\Xi_G^2} d\alpha_1 \wedge \ldots \wedge\widehat{d\alpha_i} \wedge \ldots \wedge\widehat{d\alpha_j} \wedge \ldots \wedge d\alpha_4
\end{equation}
for some linear polynomials $A_i$ in the variables $\alpha_i$. One then computes the exterior derivative
\[
d\beta = -2\frac{\sum_i A_i \frac{\partial \Xi_G}{\partial \alpha_i}}{\Xi^3} \Omega_G +  \frac{\sum_i \frac{\partial{A_i}}{\partial \alpha_i}}{\Xi^2} \Omega_G
\]
In \cite[Proposition 4.6]{G} Griffiths makes a quite general observation that whenever we have a rational form $\frac{A}{F^k} \Omega$ such that $A \in J(F)$, where $J(F)$ is the Jacobian ideal generated by partial derivatives of $F$, we can reduce the order of the pole up to an exact form. To compute $\nabla_{q_1}([\omega_G])$ one differentiates $\omega_G$ with respect to $q_1$ and writes the resulting form in terms of a basis of sections of $H^3(\Pp^3 \setminus Q)_{/S}$, which we have chosen to be $[\omega_G]$. We can use the Groebner basis of $J(\Xi_G)$ to find the $A_i$'s such that we can reduce the pole of $\frac{\partial}{\partial q_1}(\omega_G)$:
\begin{equation}
\label{relation}
\frac{\partial}{\partial q_1}(\omega_G) = \frac{-2(2(q_1+q_2)\alpha_2\alpha_4+2q_1\alpha_1\alpha_4)}{\Xi_G^3}\Omega_G = \frac{\sum_i A_i \frac{\partial \Xi_G}{\partial \alpha_i}}{\Xi_G^3} \Omega_G = \frac{1}{2}\frac{\sum_i \frac{\partial{A_i}}{\partial \alpha_i}}{\Xi_G^2} \Omega_G  - \frac{1}{2} d\beta
\end{equation}
Defining 
\begin{equation}
\label{big-b}
B(m,q) := \frac{1}{2}\sum_i \frac{\partial{A_i}}{\partial \alpha_i}
\end{equation}
which indeed depends only on masses and momenta since $A_i$ are necessarily linear, gives us the stated relation.
\end{proof}

\begin{prop} 
Let $C,D_{j,k}$ are the matrices associated to the quadratic forms $\Xi_G$, $\Xi_{G/\{e_j,e_k\}}$ respectively, and $U = C^{-1}$. Let 
\[
f_{ij}(m,q) = \frac{\sqrt{(U)_{i,j}^2 - (U)_{i,i}(U)_{j,j}}-(U)_{i,j}}{\sqrt{(U)_{i,j}^2 - (U)_{i,i}(U)_{j,j}}+(U)_{i,j}},
\]
as in Theorem 1, and $P = \frac{\sqrt{|\det D_{j,k}|}}{8\sqrt{|\det C|}}$. Then
\begin{equation}
\label{de-rham-period-via-connection}
\left[mot'_{G}, \left[\omega^{ij}_{G}\right]^{\vee}, [\omega_{G}]\right]^{\mathfrak{dr}} = P\log^{\mathfrak{dr}}(f_{ij})\mathbb{L}^{\mathfrak{dr}} .
\end{equation}
\end{prop}

\begin{proof} 
Let $G$ and $\beta$ be as in the previous lemma. We first show that $[d\beta]$, viewed as a section of $H^3_{\mathrm{dR}}(\Pp^3\setminus Q, \Delta \setminus Q \cap \Delta)_{/S}$ over the generic point, can be written as a sum of images under the face maps (\ref{face-map})
\[
\Phi_{ij,dR} : \Gamma(\Sp(k_S),(mot'_{G/\{e_i,e_j\}})_{\mathrm{dR}}) \rightarrow \Gamma(\Sp(k_S), (mot'_G)_{\mathrm{dR}}).
\]
We will always work over the generic point of the space of kinematics throughout this proof, and we drop the explicit reference to $\Sp(k_S)$ to ease the notation. 

Because $\Pp^3 \setminus Q$ is affine $[d\beta]$  is represented by the cocycle of the total complex of the double complex:
\[
(d\beta,0,\ldots,0) \in \bigoplus_{3=p+q} \bigoplus_{|I|=p} \Gamma(\Delta_I \setminus Q \cap \Delta_I, \Omega^q_{\Delta_I \setminus Q \cap \Delta_I })
\]
Then 
\begin{equation}
d_{Tot}^2(\beta,0,\ldots,0) = (d\beta,0,\ldots,0) + (0,\beta|_{\Delta_1},\beta|_{\Delta_2},\beta|_{\Delta_3},\beta|_{\Delta_4},0,\ldots,0)
\end{equation}
where $d_{Tot}^2$ is the differential of the total complex defined in (\ref{diff-of-total-complex}). As sections of $(mot'_G)_{\mathrm{dR}}$ over the generic point, we have 
\begin{equation}
\label{first-reduction}
[d\beta] = - [(0,\beta|_{\Delta_1},\beta|_{\Delta_2},\beta|_{\Delta_3},\beta|_{\Delta_4},0,\ldots,0)]
\end{equation}
Note that each $\beta|_{\Delta_i} \in \Gamma(\Delta_i \setminus Q \cap \Delta_i, \Omega^2_{\Delta_i \setminus Q \cap \Delta_i})$, but since $H^2(\Delta_i \setminus Q) \cong H^2(\Pp^2 \setminus Q) \cong 0$ by (\ref{absolute-grps}), there must exist forms $\beta_i \in \Gamma(\Delta_i \setminus Q \cap \Delta_i, \Omega^1_{\Delta_i \setminus Q \cap \Delta_i})$ such that $d\beta_i = \beta|_{\Delta_i}$. Using this fact we see that we can write for each $\beta_i$:
\begin{equation}
\label{second-reduction}
d^1_{Tot}(0,...,0,\beta_i,...,0) = (0,\ldots,\beta|_{\Delta_i},\ldots,0) - (0,...,0,r^i_{j}(\beta_i),r^i_{k}(\beta_i),r^i_{l}(\beta_i),0,...,0)
\end{equation}
where $r^i_{j}$ is the restriction of differential forms on $\Delta_i$ to forms on $\Delta_j$ with appropriate signs.
Combining (\ref{first-reduction}) and (\ref{second-reduction}) for each $i$ we get:
\begin{equation}
[d\beta] = -\left[\sum_{0 \leq i < j \leq 3}(0,...,r^i_{j}(\beta_i)+r^j_{i}(\beta_j),0,...,0)\right]
\end{equation}
Note that in the previous equation we get two contributions for each 1-face $\Delta_{\{i,j\}}$ - one from first restricting $\beta$ to the 2-face $\Delta_i$, taking the primitive of the restriction, then in turn restricting that primitive to the 1-face $\Delta_{\{i,j\}}$, and the other by the same procedure, but starting by restricting to $\Delta_j$. If we write
\begin{equation}
\label{beta-ij}
\beta_{ij} = (r^i_{j}(\beta_i)+r^j_{i}(\beta_j),0,0) \in  \bigoplus_{3=p+q} \bigoplus\limits_{\substack{|I|=p\\{i,j} \subset I}} \Gamma(\Delta_I \setminus Q \cap \Delta_I, \Omega^q_{\Delta_I \setminus Q \cap \Delta_I})
\end{equation}
where the latter is the double complex which computes $H^1(\Delta_{\{i,j\}} \setminus Q \cap \Delta_{\{i,j\}}, \Delta_{\{k,l\}} \cap \Delta_{\{i,j\}} )$, for $i,j,k,l$ pairwise distinct. Therefore we proved: $[d\beta] = - \sum_{1 \leq i < j \leq 4} \Phi_{ij}([\beta_{ij}])$. 

Since $[\beta_{ij}]$ is a section of $(mot'_{G/\{e_i,e_j\}})_{\mathrm{dR}})$ over the generic point, we can write 
\[
[\beta_{ij}] = a_{ij}(m,q)[\theta^1_{G/\{e_i,e_j\}}],
\]
where $a_{ij}(m,q) \in k_S$. For each $1 \leq i < j \leq 4$ we have equalities of motivic families of motivic periods:
\begin{equation}
\label{connection-on-de-rham-period}
\begin{split}
& \left[mot'_G, \left[\omega^{ij}_G\right]^{\vee}, \nabla_{q_1}([\omega_G])\right]^{\mathfrak{dr}} = \left[mot'_G, \left[\omega^{ij}_G\right]^{\vee}, [B(m,q)\omega_G -\frac{1}{2}d\beta]\right]^{\mathfrak{dr}} = \\
& = \left[mot'_G, \left[\omega^{ij}_G\right]^{\vee}, \left([B(m,q)\omega_G] +\frac{1}{2}\sum_{1 \leq i < j \leq 4} \Phi_{ij,dR}([\beta_{ij}]) \right)\right]^{\mathfrak{dr}} = \\
& = B(m,q)\left[mot'_G, \left[\omega^{ij}_G\right]^{\vee}, [\omega_G]\right]^{\mathfrak{dr}} + \frac{a_{ij}(m,q)}{2}\left[mot'_{G/\{e_i,e_j\}}, \left[\theta^1_{G/\{e_i,e_j\}}\right]^{\vee}, \left[\theta^1_{G/\{e_i,e_j\}}\right]\right]^{\mathfrak{dr}} = \\
& = B(m,q)\left[mot'_{G/\{e_i,e_j\}}, \left[\omega^{ij}_G\right]^{\vee}, [\omega_G]\right]^{\mathfrak{dr}} + \frac{a_{ij}(m,q)}{2}\mathbb{L}^{\mathfrak{dr}}
\end{split}
\end{equation}
where the second to last equality holds because $\Phi_{ij,dR}\left(\left[\theta^1_{G/\{e_i,e_j\}}\right]\right) = \left[\omega^{ij}_G\right]$.

On the other hand one can check that
\[
\nabla_{x}(P\log^{\mathfrak{dr}}(x)\mathbb{L}^{\mathfrak{dr}}) = \frac{\partial}{\partial x}(P)\log^{\mathfrak{dr}}(x)\mathbb{L}^{\mathfrak{dr}} + P\frac{1}{x}\mathbb{L}^{\mathfrak{dr}}.
\]
Comparing this expression with \eqref{connection-on-de-rham-period}, we see that $a_{ij}(m,q) = 2 P \frac{1}{f_{ij}}\frac{\partial f_{ij}}{\partial q_1}$, and we obtain the result up to a constant:
\[
\left[mot'_{G}, \left[\omega^{ij}_{G}\right]^{\vee}, [\omega_{G}]\right]^{\mathfrak{dr}} = P\log^{\mathfrak{dr}}(cf_{ij})
\]
The constant $c$ can be determined to be 1 by specializing to a convenient point in the space of kinematics. 
\end{proof}

\begin{rmk}
Note that an equivalence of de Rham periods induced by a face map $\Phi_{ij,dR}$ in the previous proposition is a version of Stokes' theorem for de Rham periods.
\end{rmk}
\end{subsubsection}

\begin{subsubsection}{Residues and the de Rham projection}
Since $\left[\omega^{ij}_G\right]$ and $[\omega_G]$ are sections over the generic point of $H^3_{\mathrm{dR}}(\Pp^3 \setminus Q, (\Delta_i \cup \Delta_j) \setminus Q \cap (\Delta_i \cup \Delta_j))_{/S}$, and we are interested in the de Rham period $\left[mot'_G,\left[\omega^{ij}_G\right]^{\vee},[\omega_G]\right]^{\mathfrak{dr}}$, which only depends on $\Delta_i \cup \Delta_j \subset \Delta$, we can restrict our attention to the fiber over the generic point of
\[
H^3(\Pp^3 \setminus Q, (\Delta_i \cup \Delta_j) \setminus Q \cap (\Delta_i \cup \Delta_j))_{/S}.
\]
Consider the residue map:
\[
H^3(\Pp^3 \setminus Q, (\Delta_i \cup \Delta_j) \setminus Q \cap (\Delta_i \cup \Delta_j)) \xrightarrow{\textrm{Res}} H^2(Q,(\Delta_i \cup \Delta_j) \cap Q)(-1),
\]
where we have restricted to the fiber over the generic point. Denote the former by $H_G'$, and the latter by $H_Q$. Consider the motivic period:
\[
[H_G^{\prime},[\textrm{Tube}(\sigma_{Lune})],[\omega_G]]^{\mathfrak{m}} = [H_Q,[\sigma_{Lune}],\textrm{Res}([\omega_G])]^{\mathfrak{m}}
\]
where $\sigma_{Lune}$ is the spherical lune cut out by the two hyperplanes $\Delta_i$ and $\Delta_j$ on $Q(\mathbb{C})$, and $\textrm{Tube}$ denotes the tubular neighbourhood homomorphism -- it is the transpose of the residue morphism in the Betti realization. The period of $[H_Q,[\sigma_{Lune}],\textrm{Res}(\omega_G)]^{\mathfrak{m}}$ can be computed in spherical coordinates (see also \cite[5.3.3]{Brown3}). Since we are working with a smooth quadric we can reduce to the case of a sphere $\alpha_1^2 + \alpha_2^2 + \alpha_3^2 = \alpha_4^2$, where one checks that the corresponding period is $\frac{1}{2i}\theta_{ij}$, where $\theta_{ij}$ is the angle between $\Delta_i$ and $\Delta_j$.

To show the relation of this computation with the de Rham period we are interested in, one should show that $\textrm{Res}^{\vee}\left(\left[\omega_Q^{ij}\right]\right) = \left[\omega^{ij}_G\right]^{\vee}$, where $\textrm{Res}^{\vee}$ is the dual of the residue homomorphism, and $\omega_Q^{ij}$ is a form on $Q$ with simple poles at $Q \cap \Delta_i \cap \Delta_j$. The \textit{de Rham projection} is a certain natural morphism associating to an effective period of a separated motive a de Rham period of the same motive (see \cite[\S 4]{BD}), and the de Rham period $\left[H_Q,\left[\omega_Q^{ij}\right]^{\vee},[\omega_G]\right]^{\mathfrak{dr}}$ is the image, under the de Rham projection, of the motivic period $\left[H_Q,\left[\sigma_{Lune}\right],[\omega_G]\right]^{\mathfrak{m}}$. If the previous statement about the transpose of the residue homomorphism is true, we have an equivalence of de Rham periods $\left[H_Q,\left[\omega_Q^{ij}\right]^{\vee},[\omega_G]\right]^{\mathfrak{dr}} = \left[mot'_G,\left[\omega^{ij}_G\right]^{\vee},[\omega_G]\right]^{\mathfrak{dr}}$, and have therefore associated a motivic period to the de Rham period we are interested in.
\end{subsubsection}

\begin{subsubsection}{Comparing with an analytic expression for the Feynman integral}
We are going to apply the motivic coaction to an expression for the Feynman integral of the 1-loop 4-edge graph in 4 dimensions as it appears in the physics literature \cite[Proposition 8.]{OW}, and check that it matches the coaction computed in the previous proposition. As a consequence, we will see that the coaction is both shorter and more symmetrical than the full expression of the Feynman integral in terms of dilogarithms as found in \cite{OW}, and obtain a compact expression for the arguments of the de Rham logarithms studied in this subsection.

The expression given in \cite{OW} for $I_G$ is as follows: let $C$ be the $4\times4$ matrix associated to the quadratic form $\Xi_G$, and $U = C^{-1}$. Let
\begin{equation}
\begin{split}
& \nu^0(r,s,t) := \arctan\left(\frac{C_{t,4}U_{r,4}\sqrt{|\det U|}}{U_{r,s}U_{r,4}-U_{r,r}U_{s,4}}\right)\\
& \nu^1(r,s,t) := \arctan\left(\frac{C_{t,4}U_{r,4}\sqrt{|\det U|}}{(U_{r,s}U_{r,4}-U_{r,r}U_{s,4})\sqrt{1-C_{4,4}U_{4,4}}}\right)\\
& \nu^2(r,s,t) := \arctan\left(\frac{U_{r,4}\sqrt{U_{r,r}U_{s,s}-U^2_{r,s}}}{U_{r,s}U_{r,4}-U_{r,r}U_{s,4}}\right)\\
& \nu^3(r,s,t) := \arctan\left(\frac{U_{r,4}}{\sqrt{U_{r,r}U_{4,4}-U^2_{r,4}}}\right)
\end{split}
\end{equation}
then the family of periods, depending on masses and momenta, of interest is:
\begin{equation}
\label{OW-expression}
\begin{split}
I_G = \frac{1}{16 \sqrt{|\det C|}} & \sum\limits_{\{r,s,t\} \in S_3} (2\Imm\Li_2(\exp(2\nu^0(r,s,t))) + \\ 
& + \sum\limits_{l=1}^3 (-1)^l[\Imm\Li_2(\exp(2\nu^0(r,s,t) + 2\nu^l(r,s,t))) \\
&  + \Imm\Li_2(\exp(2\nu^0(r,s,t) - 2\nu^l(r,s,t)))]) .
\end{split}
\end{equation}
Note that the above expression consists of a linear combination of 42 dilogarithms. We replace the dilogarithms with their motivic versions, and apply the coaction to each motivic dilogarithm, which reads:
\[
\Delta \Li_2^{\mathfrak{m}}(x) = \Li_2^{\mathfrak{m}}(x) \otimes \mathbb{L}^{\mathfrak{dr}} + \Li_1^{\mathfrak{m}}(x) \otimes \log^{\mathfrak{dr}}(x)\mathbb{L}^{\mathfrak{dr}} + 1 \otimes \Li_2^{\mathfrak{dr}}(x).
\]
Taking care to note that the arguments are all on the unit circle, we get a linear combination of 42 terms of the form 
\begin{equation}
\label{coaction-on-OW}
\begin{split}
& \Delta\Imm(\Li^{\mathfrak{m}}_2(z)) = \Delta \frac{1}{2i}\left(\Li^{\mathfrak{m}}_2(z) - \Li^{\mathfrak{m}}_2(\bar{z})\right) = \\
&= \frac{1}{2i}\left(\left(\Li_2^{\mathfrak{m}}(z) - \Li_2^{\mathfrak{m}}(\bar{z})\right) \otimes (\mathbb{L}^{\mathfrak{dr}})^2 + \Li_1^{\mathfrak{m}}(z) \otimes \log^{\mathfrak{dr}}(z) - \Li_1^{\mathfrak{m}}\left(\frac{1}{z}\right) \otimes \log^{\mathfrak{dr}}\left(\frac{1}{z}\right) + 1 \otimes \left(\Li^{\mathfrak{dr}}_2(z) - \Li^{\mathfrak{dr}}_2(\bar{z})\right)\right) = \\
&= \frac{1}{2i}\left(\left(\Li_2^{\mathfrak{m}}(z) - \Li_2^{\mathfrak{m}}(\bar{z})\right) \otimes (\mathbb{L}^{\mathfrak{dr}})^2 + \log^{\mathfrak{m}}\left(\frac{(1-z)^2}{z}\right) \otimes \log^{\mathfrak{dr}}(z) + 1 \otimes \left(\Li^{\mathfrak{dr}}_2(z) - \Li^{\mathfrak{dr}}_2(\bar{z})\right)\right)
\end{split}
\end{equation}

Notice that in Theorem 1 we have 6 terms of the form $\log^{\mathfrak{m}}(f_i) \otimes \log^{\mathfrak{dr}}(g_i)$ (up to prefactors), while in the coaction (\ref{coaction-on-OW}) on (\ref{OW-expression}) we have 42 such terms.

All the computations below were done in Maple. In order to check that the two expressions are equivalent, we take the 6 motivic logarithms from Theorem 1 and the 42 motivic logarithms from the coaction on (\ref{OW-expression}), and we find a basis for these 48 functions which includes the 6 motivic logarithms in Theorem 1. One can do this by applying the LLL algorithm to a matrix whose entries are the evaluations of the ($q_1$-derivative) of the 48 motivic logarithms at sufficiently many points in the space of generic kinematics. The basis found contains 27 motivic logarithms. One then checks that the relations found indeed hold on the level of the logarithms themselves and expresses the non-basis elements in terms of the basis. Plugging this back into the original expression for the coaction on $I^{\mathfrak{m}}_G$ in (\ref{OW-expression}) and collecting the de Rham logarithms with each of the 27 basis motivic logarithms, one then repeats the procedure of finding a basis for the de Rham logarithms one is left with on the right hand side of the tensor product. The basis on the de Rham side contains 20 logarithms. Expressing the non-basis de Rham logarithms in terms of these 20 and plugging this back into the previous expression with the motivic side reduced to 27 terms one observes that everything cancels out but 6 terms:
\begin{equation}
 \sum_{1 \leq j < k \leq 4}
2\log^{\mathfrak{m}}\left(\left[p_0p_1 | u_{\{j,k\}}^0u_{\{j,k\}}^1\right]\right) \otimes  \log^{\mathfrak{dr}}\left(\frac{\sqrt{(U)_{j,k}^2 - (U)_{j,j}(U)_{k,k}}-(U)_{j,k}}{\sqrt{(U)_{j,k}^2 - (U)_{j,j}(U)_{k,k}}+(U)_{j,k}}\right) 
\end{equation}
If $a_{jk}(m,q)$ are computed as in the previous proposition, and $D_{j,k}$ is the matrix associated to the quadratic form $\Xi_{G/\{e_j,e_k\}}$, we can check that
\[
a_{jk}(m,q) = \frac{\sqrt{|\det D_{j,k}}|}{4 \sqrt{|\det C|}}  \frac{1}{f_{j,k}}\frac{\partial f_{j,k}}{\partial q_1} .
\]
 As another check one can observe that $\frac{\partial }{\partial q_1} \left(\frac{1}{16 \sqrt{|\det C|}}\right) = B(m,q)$, where $B(m,q)$ is computed as in the Lemma 3.
\end{subsubsection}
\end{subsection}
\end{section}

\begin{section}{The triangle graph: motives and coaction} 
In this section we study the motives and coaction of the triangle graph, both with non-vanishing and vanishing masses.
\label{triangle}
\begin{subsection}{The triangle graph with non-vanishing masses}
\begin{subsubsection}{Motive of the triangle graph with non-vanishing masses}

In the case of the triangle graph we must work with the full graph motive because there is a linear part as well as the quadric in the polar locus of the integrand for $d=4$ dimensions. We start with the case when $F=M=N_G=3$, i.e., all masses and momenta are non-vanishing. In this case the motive of interest is:
\[
mot_G = H^2(\Pp^2 \setminus (Q \cup L), \Delta \setminus (Q \cup L) \cap \Delta)_{/S}
\]
We will need the following lemma to compute the semi-simplification of $mot_G$. We will make use of the following spectral sequence, called the \textit{Gysin spectral sequence} \cite{Deligne2}. Let $D$ be a simple normal crossing divisor in $X$, and let $D_I = D_{i_1} \cap ... \cap D_{i_k}$  where $D_i$ are the irreducible components of $D$, and $I = \{i_1,\ldots, i_k\}$, finally let $D_{\emptyset} = X$. Then we have the following spectral sequence:
\begin{equation}
\label{helful-spec-seq}
E_1^{-p,q} = \bigoplus_{|I| = p} H^{q-2p}(D_I)(-p) \Rightarrow H^{-p+q}(X \setminus D)
\end{equation}
where the $d_1^{-p,q} : E_1^{-p,q} \rightarrow E_1^{-p+1,q}$ is the alternating sum of Gysin homomorphisms
\[
H^{q-2p}(D_I)(-1) \rightarrow H^{q-2p+2}(D_{I\setminus \{i\} })
\]
for each $i \in I$, multiplied by the appropriate sign.
\begin{lemma} Let $Q \subset \Pp^2$ be a smooth quadric, and $L \subset \Pp^2$ be a projective line. Then
\[
H^1(\Pp^2 \setminus (Q \cup L)) \cong \mathbb{Q}(-1) \quad and \quad H^2(\Pp^2 \setminus (Q \cup L)) \cong \mathbb{Q}(-2).
\]
\end{lemma}
\begin{proof}
Next we need to compute $H^1(\Pp^2 \setminus (Q \cup L))$ and $H^2(\Pp^2 \setminus (Q \cup L))$, which we can do using the Gysin spectral sequence. Let us compute the relevant part of the first page. We will need the following elements:
\begin{equation}
\begin{split}
& E_1^{-2,4} \cong H^0(Q \cap L)(-2) \cong \mathbb{Q}(-2)^{\oplus 2} \\
& E_1^{-1,4} \cong H^2(Q)(-1) \oplus H^2(L)(-1) \cong \mathbb{Q}(-2)^{\oplus 2} \\
& E_1^{0,4} \cong H^4(\Pp^2) \cong \mathbb{Q}(-2) \\
& E_1^{-1,2} \cong H^0(Q)(-1) \oplus H^0(L)(-1) \cong \mathbb{Q}(-2)^{\oplus 2} \\
& E_1^{0,2} \cong H^2(\Pp^2) \cong \mathbb{Q}(-1) \\
& E_1^{0,0} \cong H^0(\Pp^2) \cong \mathbb{Q}(0)
\end{split}
\end{equation}
From this we can see that the first page of the spectral sequence is
\begin{equation}
\label{eq10}
\begin{tikzcd}[row sep=small, column sep=tiny]
0 \arrow{r} & \mathbb{Q}(-2)^{\oplus 2} \arrow{r} & \mathbb{Q}(-2)^{\oplus 2} \arrow {r} & \mathbb{Q}(-2) \\
0 \arrow{r} & 0  \arrow{r} & 0 \arrow {r} & 0 \\
0 \arrow{r} & 0 \arrow{r} & \mathbb{Q}(-1)^{\oplus 2} \arrow{r} & \mathbb{Q}(-1) \\
0 \arrow{r} & 0  \arrow{r} & 0 \arrow {r} & 0 \\
0 \arrow{r} & 0  \arrow{r} & 0 \arrow {r} & \mathbb{Q}(0) \\
\end{tikzcd}
\end{equation}

Let $v_L$ and $v_Q$ be generators of $H^0(L)(-1)$ and $H^0(Q)(-1)$ respectively. Then on the second row we have the Gysin morphisms $H^0(L)(-1) \xrightarrow{G} H^2(\Pp^2)$, and $H^0(Q)(-1) \xrightarrow{G} H^2(\Pp^2)$, the images of which are the fundamental classes $G(v_L)=[L]$, $G(v_Q)=[Q]$, respectively. Moreover, $[L]$ generates $H^2(\Pp^2)$, and $[Q]=2[L]$, since $Q$ intersects $L$ in two points, generically. Therefore, the morphism $\mathbb{Q}(-1)^{\oplus 2} \rightarrow \mathbb{Q}(-1)$ is surjective, and the kernel is generated by $(2v_L,-v_Q)$. Hence $gr^W_2H^2(\Pp^2 \setminus Q \cup L) \cong 0$ and $gr^W_2H^1(\Pp^2 \setminus Q \cup L) \cong \mathbb{Q}(-1)$. 

Finally, $gr^W_4H^2(\Pp^2 \setminus Q \cup L)$ is computed by the kernel of the map $E^1_{-2,4} \rightarrow E^1_{-1,4}$. The kernel of $H^0(Q \cap L)(-2) \rightarrow H^2(L)(-1)$ is one-dimensional becuase $H^2(L \setminus Q \cap L) = 0$, and similarly the kernel of the other Gysin map $H^0(Q \cap L)(-2) \rightarrow H^2(Q)(-1)$ is one-dimensional. Hence $gr^W_4H^2(\Pp^2 \setminus Q \cup L) \cong \mathbb{Q}(-2)$. 
\end{proof}
\begin{lemma}
\label{triangle-general}
The weight graded pieces of the one-loop triangle graph motive $mot_G$ with non-vanishing masses and momenta are
\[
\text{gr}^Wmot_G \cong \mathbb{Q}(-2)_{/S} \oplus \mathbb{Q}(-1)^{\oplus 5}_{/S} \oplus \mathbb{Q}(0)_{/S}
\]
\end{lemma}
\begin{proof}
It is enough to prove the claim over any fibre $t \in S(\mathbb{C})$. The $E_1$ page of the relative cohomology spectral sequence reads as follows:
\begin{equation}
\label{eq11}
\begin{tikzcd}[row sep=small, column sep=small]
H^2(\Pp^2 \setminus (Q \cup L)) \arrow{r} & \bigoplus_{|I|=1} H^2(\Delta_I \setminus (Q \cup L) \cap \Delta_I)  \arrow{r} & \bigoplus_{|I|=2} H^2(\Delta_I \setminus (Q \cup L) \cap \Delta_I) \\
H^1(\Pp^2 \setminus (Q \cup L)) \arrow{r} & \bigoplus_{|I|=1} H^1(\Delta_I \setminus (Q \cup L) \cap \Delta_I)  \arrow{r} & \bigoplus_{|I|=2} H^1(\Delta_I \setminus (Q \cup L) \cap \Delta_I) \\
H^0(\Pp^2 \setminus (Q \cup L)) \arrow{r} & \bigoplus_{|I|=1} H^0(\Delta_I \setminus (Q \cup L) \cap \Delta_I)  \arrow{r} & \bigoplus_{|I|=2} H^0(\Delta_I \setminus (Q \cup L) \cap \Delta_I)
\end{tikzcd}
\end{equation}

To compute this, we note that $H^1(\Delta_i \setminus (Q \cup L) \cap \Delta_i) \cong H^1(\Pp^1 \setminus \{u_{1_i},u_{2_i},l_{1_i}\}) \cong \mathbb{Q}(-1)^{\oplus 2}$, where $u_{1_i},u_{2_i},l_{1_i}$ are the points of intersection of the quadric $Q$ and the line $L$ with the face $\Delta_i \hookrightarrow \Delta$. From the previous lemma we get the leftmost column, and note that $E_1^{1,2}=E_1^{2,1}=E_1^{1,2}=E_1^{2,2}=0$. Taking cohomology of the rows we get the result.
\end{proof}
\end{subsubsection}
\begin{subsubsection}{Coaction on the triangle graph with non-vanishing masses and momenta}
\label{triangle-coaction-non-vanishing}
Notice from the proofs of the previous two lemmas that all the motivic periods of the triangle graph motive in the case of non-vanishing masses, except the one of weight 4 and one of weight 0, are equivalent to motivic periods of the motives associated to the faces of $\Delta$ via the face maps, as in (\ref{coaction-4-edge-computation}). Note that $G/e_i$, for $i \in \{1,2,3\}$, are bubble graphs with non-vanishing masses. In this case we consider the full motive of the bubble graph $mot_{G/e_i}$ (\ref{bubble-motive}) for all $i$. The fiber over the generic point of $(gr^W_2mot_{G/e_i})_{\mathrm{dR}}$ is of rank 2 and we can choose its basis to be the classes of the two forms:
\begin{equation}
\label{omega-ij-2}
\begin{split}
\theta^1_{G/e_i} & = \frac{\alpha_kd\alpha_j - \alpha_jd\alpha_k}{\Xi_{G/e_i}} \\
\theta^2_{G/e_i} & = \frac{(x-y)(\alpha_kd\alpha_j - \alpha_jd\alpha_k)}{(\alpha_k-x\alpha_j)(\alpha_k-y\alpha_j)} \\
\end{split}
\end{equation}
where $i,j,k$ are pairwise distinct and $x,y$ are coordinates of points $u_1 \in Q \cap \Delta_I$ and $L|_{\Delta_i}$ respectively in the coordinate chart of $\Delta_i$ where $\alpha_j \not = 0$.

We can choose a basis of the fiber of $(mot_G)_{\mathrm{dR}}$ over the generic point such that it contains the class of the Feynman form of $G$ in 4 dimensions $[\omega_G]$, and the 5 weight 2 classes denoted by $\left[\omega^j_i\right]$, which we define to be the images of $\left[\theta^j_{G/e_i}\right]$ under the de Rham component of the face map (\ref{face-map}):
\begin{equation}
\begin{split}
\Phi_{i,dR} & : (mot_{G/e_i})_{\mathrm{dR}} \rightarrow (mot_G)_{\mathrm{dR}} \\
& \left[\theta^j_{G/e_i}\right] \mapsto \left[\omega^j_i\right]
\end{split}
\end{equation}
Note that there are 6 such classes $\left[\omega^j_i\right]$, two for each face $\Delta_i$, but there is a relation between them, as can be seen in the proof of Lemma \ref{triangle-general}.

\begin{prop}
Let $G$ be a triangle graph with non-vanishing generic masses and momenta. Then the motivic Galois coaction on the associated motivic Feynman amplitude is:
\begin{equation}
\label{coaction-3-gen}
\begin{split}
\Delta I_G^{\mathfrak{m}} & = I_{G/e_1}^{\mathfrak{m}}\left(\theta^1_{G/e_1}\right) \otimes \left[mot_G, \left[\omega^1_1\right]^{\vee},[\omega_G]\right]^{\mathfrak{dr}} + \sum\limits_{i=2,3} \sum_{j=1,2} I_{G/e_i}^{\mathfrak{m}}\left(\theta^j_{G/e_i}\right) \otimes \left[mot_{G}, \left[\omega^j_i\right]^{\vee},[\omega_G]\right]^{\mathfrak{dr}} \\
&+ I_G^{\mathfrak{m}} \otimes (\mathbb{L}^{\mathfrak{dr}})^2 + 1\otimes I_G^{\mathfrak{dr}}
\end{split}
\end{equation}
\end{prop}

\begin{rmk}
Both the motivic and the de Rham side of the coaction in the previous proposition can be expressed in terms of motivic and de Rham logarithms respectively, using the same techniques as in the previous section.
\end{rmk}

\end{subsubsection}
\end{subsection}

\begin{subsection}{Triangle graph with vanishing masses}
\begin{subsubsection}{Motive}
When one of the masses $m_i$ vanishes, i.e., when we have $F=3$ and $M=2$, the quadric $Q$ passes through a point of $\Delta$ defined by the vanishing of the coordinates corresponding to the other two edges $V(\alpha_j) \cap V(\alpha_k)$, where $i,j,k$ are pairwise distinct. In the following figure, we have chosen $m_1=0$:

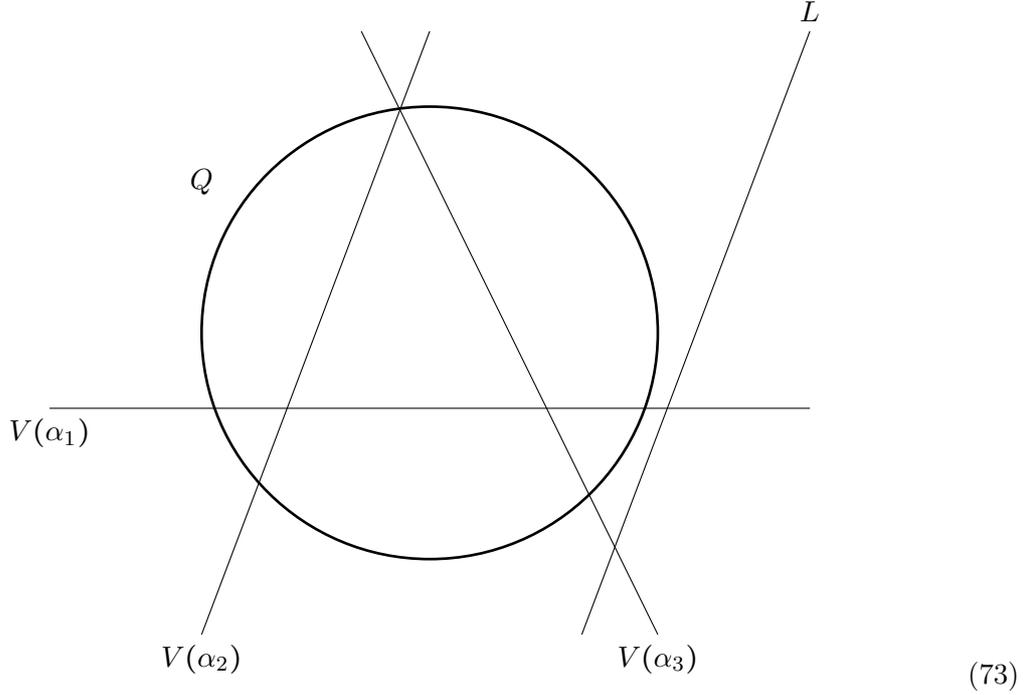
\begin{equation}
\begin{tikzpicture}
\draw (-3,4) node {$Q$};
\draw[line width=1pt] (0,2) circle (3cm);
\draw (-3,-2) node[below] {$V(\alpha_2)$}  -- (0,6);
\draw (-5,1) node[below] {$V(\alpha_1)$} -- (5,1);
\draw (3,-2) node[below] {$V(\alpha_3)$} -- (-0.9,6);
\draw (5,6) node[above] {$L$} -- (2,-2);
\end{tikzpicture}
\end{equation}

This means that the poles of the integrand meet the boundary of the domain of integration, and $\Delta \cup Q \cup L$ is not simple normal crossing anymore, over each fiber of $K^{gen}_{F,M}$. Thus we cannot realize the relevant Feynman integral of the graph as a family of motivic periods of the motive $H^2(\Pp^2 \setminus (Q \cup L), \Delta \setminus (Q \cup L) \cap \Delta)_{/S}$. To remedy the situation we must blow up along $V(\alpha_j) \cap V(\alpha_k) \subset \Pp^2$. The point $V(\alpha_j) \cap V(\alpha_k)$ corresponds to a motic subgraph $\gamma$ (see definition \ref{motic-subgraphs}) spanned by the edges $\{e_j,e_k\}$ of our triangle graph $G$.

Let $\pi_G : P^G \rightarrow \Pp^2$ for the blow-up of $\Pp^2$ at the point $V(\alpha_3) \cap V(\alpha_2) = [1:0:0]$. Denote by $\widetilde{Q},\widetilde{L}$ the strict transforms of $Q$ and $L$ where these are given by the vanishing of:
\begin{equation}
\begin{split}
&\Psi_G = \alpha_1 + \alpha_2 + \alpha_3, \text{ and} \\
&\Xi_G(m,q) =  q_1^2\alpha_2\alpha_3 + q_2^2\alpha_1\alpha_3+q_3^2\alpha_1\alpha_2 + (m_2^2\alpha_2+m_3^2\alpha_3)\Psi_G .
\end{split}
\end{equation}
Let $D = \pi_G^{-1}(\Delta)$ be the total transform of $\Delta$, and denote by $D_{-1}$ the exceptional divisor corresponding to the motic subgraph $\gamma$ (see (\ref{fig1})). By the results recalled in \ref{motivic-lifts} we get
\begin{equation}
\begin{split}
&[\widetilde{\sigma_G}] \in \Gamma(U^{gen}_{2,3}, (mot_G)_B^{\vee}) \text{, where} \\
&mot_G = (H^2_B(P^G \setminus \widetilde{Q} \cup \widetilde{L}, D \setminus (\widetilde{Q} \cup \widetilde{L}) \cap D)_{/S},H^2_{\mathrm{dR}}(P^G \setminus \widetilde{Q} \cup \widetilde{L}, D \setminus (\widetilde{Q} \cup \widetilde{L}) \cap D)_{/S},c).
\end{split}
\end{equation}

\begin{lemma}
\label{pullback-lemma}
$\pi_G^*(\omega_G(m,q))$ does not have any poles along the exceptional divisor $D_{-1}$. Hence $\pi_G^*(\omega_G(m,q))$ is a global section of $\Omega^2_{P^G \setminus \widetilde{Q} \cup \widetilde{L}/k_S}$, and defines a class $[\pi_G^*(\omega_G(m,q))] \in (mot_G)_{\mathrm{dR}}$.
\end{lemma}
\begin{proof}
We can consider the following affine charts of $P^G$: let $\mathbb{A}_{23,1}$ be the affine space with the coordinate ring $\mathcal{O}(\mathbb{A}_{23,1}) = \mathbb{Z}[\beta^{23,1}_1,\beta^{23,1}_2]$, where $\beta^{23,1}_1 = \frac{\alpha_1}{\alpha_2}$, and $\beta^{23,1}_2 = \alpha_2$, and $\alpha_3=1$. The exceptional divisor, in this affine chart, is given by $\beta^{23,1}_2=0$. One analogously defines $\mathbb{A}_{23,2}$, by $\beta^{23,2}_1 = \alpha_1$, and $\beta^{23,2}_2 = \frac{\alpha_2}{\alpha_1}$, where $\alpha_3=1$. Two more affine spaces, given by $\alpha_1=1$ and $\alpha_2=1$, away from the exceptional divisor, complete a covering of $P^G$.

Let us see what $\pi_G^*(\omega_G(m,q))$ looks like in $\mathbb{A}_{23,1}$. The differential form $\Omega_G$ is pulled back to $\pi_G^*(\Omega_G) = d(\beta^{23,1}_1\beta^{23,1}_2)d\beta^{23,1}_2 = \beta^{23,1}_2d\beta^{23,1}_1\beta^{23,1}_2$. We also have
\begin{equation}
\begin{split}  
&\Psi_{23}:=\pi_G^*(\Psi_G) = \beta^{23,1}_2d\beta^{23,1}_1 + \beta^{23,1}_2 + 1 \\
&\Xi_{23}:=\pi_G^*(\Xi_G) = \beta^{23,1}_2(q^2_1 + q_2^2\beta^{23,1}_1 + q^2_3\beta^{23,1}_1\beta^{23,1}_2 + m_2^2\beta^{23,1}_2\Psi_{23}) \\
\end{split}
\end{equation}
and therefore
\[
\pi_G^*(\omega_G(m,q))=\left(\frac{\beta^{23,1}_2d\beta^{23,1}_1\beta^{23,1}_2}{\Psi_{23}\Xi_{23}}\right)
\]
has no poles along $\beta^{23,1}_2$ since it cancels out. The result follows.

\end{proof} 

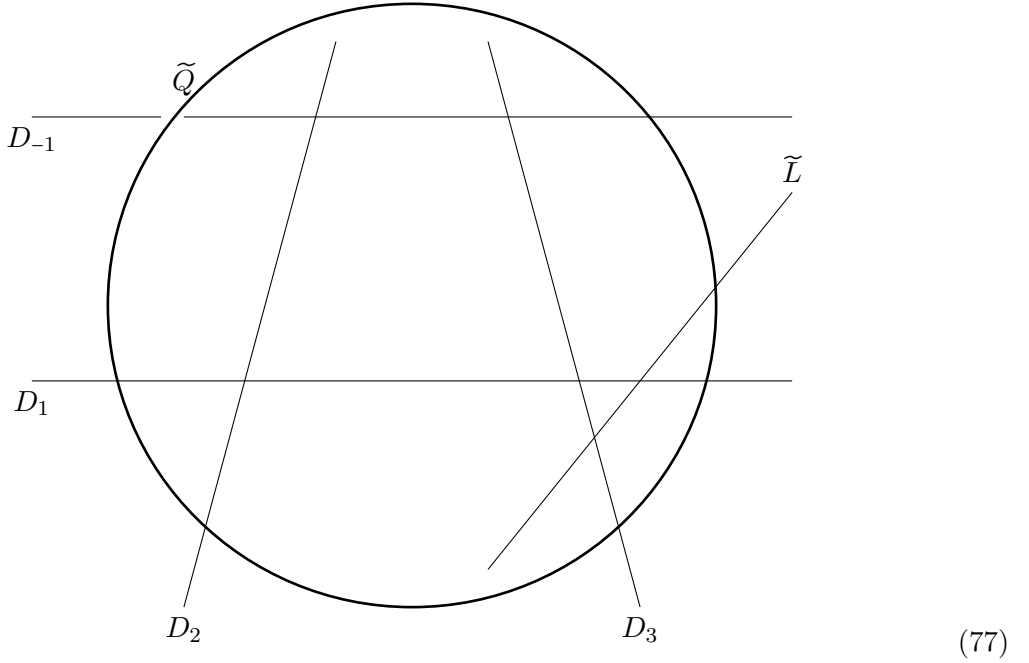
\begin{equation}
\label{fig1}
\begin{tikzpicture}
\draw (-3,5) node {$\widetilde{Q}$};
\draw[line width=1pt] (0,2) circle (4cm);
\draw (-3,-2) node[below] {$D_2$}  -- (-1,5.5);
\draw (-5,1) node[below] {$D_1$} -- (5,1);
\draw (-3,4.5) --  (5,4.5);
\draw (-5,4.5) node[below] {$D_{-1}$} -- (-3.3,4.5);
\draw (3,-2) node[below] {$D_3$} -- (1,5.5);
\draw (5,3.5) node[above] {$\widetilde{L}$} -- (1,-1.5);
\end{tikzpicture}
\end{equation}
When either of the other two masses, $m_2$ or $m_3$, vanishes, we have to blow up points $V(\alpha_1)\cap V(\alpha_3)$, and $V(\alpha_1)\cap V(\alpha_2)$ respectively. Denote the corresponding exceptional divisors $D_{-2},D_{-3}$. The proof of the previous lemma proceeds analogously, for each of the exceptional divisors. Having done this we get the associated motivic Feynman amplitude:
\[
\left[mot_G, \left[\sigma_G\right], \left[\pi^*(\omega_G)\right]\right]^{\mathfrak{m}} \in \mathcal{P}^{\mathfrak{m},U^{gen}_{k,3}, gen}_{\mathcal{H}(S)},
\]
where $1\leq k < 3$ is the number of non-vanishing masses.

\begin{lemma}
\label{triangle-vanishing}
Let $G$ be a 1-loop graph with 3 internal edges. Let $ 1 \leq v \leq 3$ be the number of vanishing masses. Then 
\[
\text{gr}^Wmot_G \cong \mathbb{Q}(-2)_{/S} \oplus \mathbb{Q}(-1)_{/S}^{\oplus 5 - v} \oplus \mathbb{Q}(0)_{/S}
\]
\end{lemma}
\begin{proof}
Let us see how one proceeds in the case of one vanishing mass. Without loss of generality let that mass be $m_1 = 0$. 

We apply the relative cohomology spectral sequence at each fiber $t \in S(\mathbb{C})$:
\[
H^2(P^G \setminus (\widetilde{Q} \cup \widetilde{L}), D \setminus (\widetilde{Q} \cup \widetilde{L}) \cap D)
\]
Its first page reads:
\begin{equation}
\label{eq12}
\begin{tikzcd}[column sep = small, row sep=small]
H^2(P^G \setminus (\widetilde{Q} \cup \widetilde{L})) \arrow{r} & \bigoplus\limits_{i \in \{-1,1,2,3\}} H^2(D_i \setminus (\widetilde{Q} \cup \widetilde{L}) \cap D_i)  \arrow{r} & \bigoplus\limits_{i \in \{-1,1,2,3\}} H^2(D_{ij} \setminus (\widetilde{Q} \cup \widetilde{L}) \cap D_{ij})\\
H^1(P^G \setminus (\widetilde{Q} \cup \widetilde{L})) \arrow{r} & \bigoplus\limits_{i \in \{-1,1,2,3\}} H^1(D_i \setminus (\widetilde{Q} \cup \widetilde{L}) \cap D_i) \arrow{r} & \oplus_{i \in \{-1,1,2,3\}} H^1(D_{ij} \setminus (\widetilde{Q} \cup \widetilde{L}) \cap D_{ij})\\
H^0(P^G \setminus (\widetilde{Q} \cup \widetilde{L})) \arrow{r} & \bigoplus\limits_{i \in \{-1,1,2,3\}} H^0(D_i \setminus (\widetilde{Q} \cup \widetilde{L}) \cap D_i) \arrow{r} & \bigoplus\limits_{i \in \{-1,1,2,3\}} H^0(D_{ij} \setminus (\widetilde{Q} \cup \widetilde{L}) \cap D_{ij})
\end{tikzcd}
\end{equation}
We use the spectral sequence (\ref{helful-spec-seq}) to compute $H^2(P^G \setminus (\widetilde{Q} \cup \widetilde{L}))$ and $H^1(P^G \setminus (\widetilde{Q} \cup \widetilde{L}))$.

\begin{equation}
\label{vanishing-masses-gysin}
\begin{tikzcd}[row sep=small, column sep=tiny]
\ldots \arrow{r} & H^2(\widetilde{Q})(-1) \bigoplus H^2(\widetilde{L})(-1) \arrow{r} & H^4(P^G) \\
\ldots \arrow{r} & 0  \arrow{r} & 0 \\
\ldots \arrow{r} & H^0(\widetilde{Q})(-1) \bigoplus H^0(\widetilde{L})(-1) \arrow {r} & H^2(P^G) \\
\ldots \arrow{r} & 0  \arrow{r} & 0 \\
\ldots \arrow{r} & 0 \arrow {r} & H^0(P^G)
\end{tikzcd}
\end{equation}

We have that $H^2(P^G)$ is generated by $[\widetilde{L}], [D_{-1}]$, and the image of the Gysin morphism $H^0(\widetilde{Q})(-1) \rightarrow H^2(P^G)$ is $[Q] = 2[\widetilde{L}] - [D_{-1}]$ because the intersection of $Q \cdot D_{-1} = -1$ and $Q \cdot L = 2$. The image of $H^0(\widetilde{L})(-1) \rightarrow H^2(P^G)$ is simply $[\widetilde{L}]$, so the difference of these two morphisms, which is the map, $E^1_{-1,2} \rightarrow E^1_{0,2}$, is an isomprhism. Therefore $gr^W_2 H^1(P^G \setminus (\widetilde{Q} \cup \widetilde{L})) = 0 $ and $gr^W_2 H^2(P^G \setminus (\widetilde{Q} \cup \widetilde{L})) = 0$. In the top row nothing changes by blowing up, hence the computation is the same as in the previous lemma. We get $H^2(P^G \setminus (\widetilde{Q} \cup \widetilde{L})) \cong \mathbb{Q}(-2)$ and $H^1(P^G \setminus (\widetilde{Q} \cup \widetilde{L})) \cong 0$.

Next, we compute $\bigoplus\limits_{i \in \{-1,1,2,3\}} H^1(D_i \setminus (\widetilde{Q} \cup \widetilde{L}) \cap D_i)$. For the exceptional divisor we have $D_{-1} \cap (\widetilde{Q} \cup \widetilde{L}) = \{1\ \text{ point} \}$, therefore $H^1(D_{-1} \setminus (\widetilde{Q} \cup \widetilde{L}) \cap D_{-1}) \cong H^1(\mathbb{A}^1) = 0$. For $D_1 = \widetilde{V(\alpha_1)}$, since it is away from the exceptional divisor,  we have that $D_{1} \cap (\widetilde{Q} \cup \widetilde{L}) \cong D_{1} \cap (Q \cup L) = \{3 \text{ points}\}$, therefore $H^1(D_1 \setminus (\widetilde{Q} \cup \widetilde{L}) \cap D_1) \cong \mathbb{Q}(-1)^{\oplus 2}$. For the remaining two irreducible components $D_3$ and $D_2$ we have $D_{3/2} \cap (\widetilde{Q} \cup \widetilde{L}) = \{2 \text{ points} \}$, therefore $H^1(D_{3/2} \setminus (\widetilde{Q} \cup \widetilde{L}) \cap D_{3/2}) \cong \mathbb{Q}(-1)$. Putting this together into (\ref{eq12}) we get
\[
\text{gr}^WH^2(P^G \setminus (\widetilde{Q} \cup \widetilde{L}), D \setminus (\widetilde{Q} \cup \widetilde{L}) \cap D) \cong \mathbb{Q}(-2) \bigoplus \mathbb{Q}(-1)^{\bigoplus 4} \bigoplus \mathbb{Q}(0).
\]

When another mass vanishes, say $m_2$, in the spectral sequence \eqref{vanishing-masses-gysin} we get $H^2(P^G) \cong \mathbb{Q}(-1)^{\bigoplus 3}$, generated by $[\widetilde{L}],[D_{-1}],[D_{-2}]$. The image of the Gysin morphism $H^0(\widetilde{Q})(-1) \rightarrow H^2(P^G)$ is generated by $2[\widetilde{L}] - [D_{-1}] - [D_{-2}]$. Therefore the rank of the cokernel of the morpshim $E^1_{-1,2} \rightarrow E^1_{0,2}$ is of rank 1, and the kernel remains trivial. Hence we have $gr^W_2H^2(P^G \setminus (\widetilde{Q} \cup \widetilde{L})) \cong \mathbb{Q}(-1)$. Everything else remains the same as in the previous case, therefore we have $H^2(P^G \setminus (\widetilde{Q} \cup \widetilde{L})) \cong \mathbb{Q}(-2) \oplus \mathbb{Q}(-1)$, and $H^1(P^G \setminus (\widetilde{Q} \cup \widetilde{L})) \cong 0$.

We also have
\[
\bigoplus\limits_{i \in \{-2,-1,1,2,3\}} H^1(D_i \setminus (\widetilde{Q} \cup \widetilde{L}) \cap D_i) \cong \mathbb{Q}(-1)^{\oplus 2}
\]
where one $\mathbb{Q}(-1)$ comes from $D_3$ and $D_1$ each.
Plugging these results into the spectral sequence \eqref{eq12} we get:
\[
\text{gr}^WH^2(P^G \setminus (\widetilde{Q} \cup \widetilde{L}), D \setminus (\widetilde{Q} \cup \widetilde{L}) \cap D) \cong \mathbb{Q}(-2) \bigoplus \mathbb{Q}(-1)^{\bigoplus 3} \bigoplus \mathbb{Q}(0).
\]
Finally, if all three masses vanish we calculate, analogously to the previous two cases, $H^2(P^G \setminus (\widetilde{Q} \cup \widetilde{L})) \cong \mathbb{Q}(-2) \oplus \mathbb{Q}(-1)^{\oplus 2}$, and $H^1(P^G \setminus (\widetilde{Q} \cup \widetilde{L})) \cong 0$, where $P^G$ is the blow up of $\Pp^2$ at 3 points. This case differs from the previous ones only in $gr^W_2 H^2(P^G \setminus (\widetilde{Q} \cup \widetilde{L}))$, which is given by the cokernel of the morphism $E^1_{-1,2} \rightarrow E^1_{0,2}$ in \eqref{vanishing-masses-gysin}. We have that $E^1_{0,2} \cong H^2(P^G)$ generated by $[\widetilde{L}],[D_{-1}],[D_{-2}],[D_{-3}]$, and the image of the Gysin morphism $H^0(\widetilde{Q})(-1) \rightarrow H^2(P^G)$ is generated by $2[\widetilde{L}] - [D_{-1}] - [D_{-2}] - [D_{-3}]$. Hence the cokernel is of rank 2. Moreover, we have: $\bigoplus\limits_{i \in \{-3, -2,-1,1,2,3\}} H^1(D_i \setminus (\widetilde{Q} \cup \widetilde{L}) \cap D_i) \cong 0$ because we remove 1 point from each of the three exceptional divisors, its intersection with $\tilde{Q}$, and one from each of the $D_i$, their intersection with $\tilde{L}$, which makes $D$ into a hexagon of affine lines. Therefore, from \eqref{eq12} we get:
\[
\text{gr}^WH^2(P^G \setminus (\widetilde{Q} \cup \widetilde{L}), D \setminus (\widetilde{Q} \cup \widetilde{L}) \cap D) \cong \mathbb{Q}(-2) \bigoplus \mathbb{Q}(-1)^{\bigoplus 2} \bigoplus \mathbb{Q}(0).
\]
\end{proof}

\end{subsubsection}
\begin{subsubsection}{Coaction on the triangle graph with two or more vanishing masses}

When one of the masses in the triangle graph vanishes, and other masses and momenta are generic, the proof of lemma \ref{triangle-vanishing} and the same argument as in \ref{triangle-coaction-non-vanishing} implies that the coaction in this case becomes:
\begin{equation}
\label{coaction-3-1mass0}
\begin{split}
\Delta I_G^{\mathfrak{m}} & = \left[mot^{\prime \prime}_{G/e_2}, \left[\sigma_{G/e_2}\right],\left[\theta^2_{G/e_1}\right]\right]^{\mathfrak{m}} \otimes \left[mot_G, \left[\omega^2_2\right]^{\vee},[\omega_G]\right]^{\mathfrak{dr}} + \\
& + \left[mot^{ \prime\prime}_{G/e_3}, \left[\sigma_{G/e_3}\right], \left[\theta^2_{G/e_3}\right]\right]^{\mathfrak{m}} \otimes \left[mot_G,\left[\omega^2_3\right]^{\vee} ,[\omega_G]\right]^{\mathfrak{dr}} + \\
&+ \sum_{j=1,2} I_{G/e_1}^{\mathfrak{m}}\left(\theta^j_{G/e_1}\right) \otimes \left[mot_G,\left[\omega^j_1\right]^{\vee},[\omega_G]\right]^{\mathfrak{dr}} + \\
& + I_G^{\mathfrak{m}} \otimes (\mathbb{L}^{\mathfrak{dr}})^2 + 1\otimes I_G^{\mathfrak{dr}}
\end{split}
\end{equation}
where we have taken, without loss of generality, $m_1=0$, and $mot^{\prime \prime}_{G/e_2}$ and $mot^{\prime \prime}_{G/e_3}$ are motives of bubble graphs with one vanishing mass (\ref{bubble-motive-1mass0}), while $mot_{G/e_1}$ is the motive of a bubble graph with non-vanishing masses (\ref{bubble-motive}).

In the proof of Lemma \ref{triangle-vanishing} we saw that when 2 masses of the triangle graph vanish we have 
\[
gr^WH^2(P^G \setminus (\tilde{Q} \cup \tilde{L})) \cong \mathbb{Q}(-2) \oplus \mathbb{Q}(-1).
\] 
Similarly when all 3 masses vanish we get 
\[
gr^WH^2(P^G \setminus (\tilde{Q} \cup \tilde{L})) \cong \mathbb{Q}(-2) \oplus \mathbb{Q}(-1)^{\oplus 2}.
\]
Note that for the motive of the triangle graph with generic non-vanishing masses and momenta all motivic periods of weight 2 are equivalent to motivic periods of the motives $H^1(\Delta_i \setminus (Q \cup L) \cap \Delta_i, \Delta_i \cap (\Delta_j \bigcup \Delta_K))$, where $1\leq i \leq 3$, and $i,j,k$ pairwise distinct, via the face maps. This enables us to write the motivic side of the coaction in terms of the motivic periods of quotient graphs. However, when two or more masses vanish, this is not the case any longer.

We can still compute the coaction, as we now show in the case of all three masses vanishing and three non-trivial external momenta, by using the residue. Consider the following hyperplanes:
\[
F_1 := V(\alpha_2 + \alpha_3), \quad F_2 := V(\alpha_1 + \alpha_3)
\]
and denote by $F_i$ their strict transforms as well. Note that we still consider all schemes base changed to $K_{F,M}$. Let $D_{-i}$ be the exceptional divisor over the point $V(\alpha_j)\cap V(\alpha_k)$ for $1\leq i,j,k \leq 3$ pairwise distinct. Denote the following intersection points:
\[
D_{i} \cap \tilde{L} = l_i, \quad 
D_{-i} \cap \tilde{Q} = u_i, \quad
F_i \cap D_{-i} = g_i, \quad 
F_i \cap D_i = t_i, \quad \text{ for } i=1,2
\]
and
\[
D_{-i} \cap D_j = p_0^i, \quad D_{-i} \cap D_k = p_1^i, \quad i=1,2, \quad i,j,k \text{ pairwise distinct}
\]
as well as
\[
\tilde{Q} \cap \tilde{L} = \{f_0,f_1\}, \quad D_i \cap \tilde{L} = d_i \quad i=1,2 .
\]

\begin{thm}
Let $G$ be the triangle graph with all internal masses vanishing, and non-trivial momenta. 
Then the coaction on the associated motivic Feynman amplitude is:
\begin{equation}
\label{coaction-3-3mass0}
\begin{split}
\Delta I_G^{\mathfrak{m}} & = \left(a_1 \log^{\mathfrak{m}}([g_1u_1|p^1_0p^1_1]) + a_2 \log^{\mathfrak{m}}([g_2u_2|p^2_0p^2_1]) \right) \otimes \left(\log^{\mathfrak{dr}}(f_0f_1|d_1d_2) + \log^{\mathfrak{dr}}(f_0f_1|d_1d_3)\right) +  \\
& + I_G^{\mathfrak{m}} \otimes (\mathbb{L}^{\mathfrak{dr}})^2 + 1\otimes I_G^{\mathfrak{dr}}
\end{split}
\end{equation}
where $a_1,a_2$ are undetermined constants in $k_S$.
The cross-ratios evaluate to:
\[
\log^{\mathfrak{m}}([e_iu_i|p^i_0p^i_1]) = \log^{\mathfrak{m}}\left(\frac{q_j^2}{q_k^2}\right), \quad i=1,2, \quad i,j,k \text{ pairwise distinct}
\]
and
\[
\log^{\mathfrak{dr}}([f_0f_1|d_1d_2]) = \log^{\mathfrak{dr}}\left(\frac{\left(q_1^2 + q_2^2-q_3^2 + \sqrt{q_1^4+q_2^4+q_3^4-2q_1^2q_3^2-2q_2^2q_3^2}\right)^2}{4q_1^2q_2^2}\right)
\]
and 
\[
\log^{\mathfrak{dr}}([f_0f_1|d_1d_3]) = \log^{\mathfrak{dr}}([f_0f_1|d_1d_2]) + \log^{\mathfrak{dr}}\left(\frac{q_1^2 + q_3^2-q_2^2 - \sqrt{q_1^4+q_2^4+q_3^4-2q_1^2q_3^2-2q_2^2q_3^2}}{q_1^2 + q_3^2-q_2^2 + \sqrt{q_1^4+q_2^4+q_3^4-2q_1^2q_3^2-2q_2^2q_3^2}}\right)
\]
\end{thm}

\begin{proof}
From Lemma \ref{triangle-vanishing} and the definition of the coaction we have
\[
\Delta I_G^{\mathfrak{m}} = I_G^{\mathfrak{m}} \otimes (\mathbb{L}^{\mathfrak{dr}})^2 + \sum_{i=1,2} [mot_G, [\sigma_G], e_i]^{\mathfrak{m}} \otimes \left[mot_G, e_i^{\vee}, \omega_G\right]^{\mathfrak{m}} + 1\otimes I_G^{\mathfrak{dr}},
\]
where $e_i$ are weight 2 elements of a chosen basis of $(mot_G)_{\mathrm{dR}}$.
We restrict to the fiber over the generic point. To identify the de Rham side we consider the residue morphism along $\tilde{L}$:
\[
\mathrm{Res}_{\tilde{L}} : H^2(P^G \setminus (\tilde{Q} \cup \tilde{L}), D \setminus (\tilde{Q} \cup \tilde{L}) \cap D) \rightarrow H^1(\tilde{L} \setminus \tilde{L} \cap \tilde{Q}, \tilde{L} \cap D) \cong H^1(\mathbb{P}^1 \setminus \{f_0,f_1\},\{d_1,d_2,d_3\})
\]
This morphism gives us an equivalence of de Rham periods, and by computing the motivic periods of the motive on the right hand side, and their de Rham projections (see \ref{examples}), we get that each de Rham period in our coaction is a linear combination of $\log^{\mathfrak{dr}}([f_0f_1|d_1d_2])$ and $\log^{\mathfrak{dr}}([f_0f_1|d_1d_3])$.

To determine the motivic periods in the coaction consider the pullback morphism:
\[
i^* : H^2(P^G \setminus (\tilde{Q} \cup \tilde{L}), D \setminus (\tilde{Q} \cup \tilde{L}) \cap D) \rightarrow H^2(P^G \setminus (\tilde{Q} \cup \tilde{L} \cup \cup_{i=1,2,3}F_i),D \setminus (\tilde{Q} \cup \tilde{L} \cup \cup_{i=1,2}F_i))
\]

Computing the Gysin spectral sequence \eqref{helful-spec-seq} and the relative cohomology spectral sequence for the motive on the right we get that:
\[
H^2(P^G \setminus (\tilde{Q} \cup \tilde{L} \cup \cup_{i=1,2}F_i) \cong \mathbb{Q}(-2)
\]
and 
\[
H^1(D_j \setminus (\tilde{Q} \cup \tilde{L} \cup \cup_{i=1,2}F_i)) \cong \mathbb{Q}(-1) \quad \text{for } j=\{-1,-2\}
\]
and
\[
H^1(D_j \setminus (\tilde{Q} \cup \tilde{L} \cup \cup_{i=1,2}F_i)|_{D_j}) \cong 0 \quad \text{for } j=\{1,2\},
\]
therefore, since the pullback preserves the weight, the class of $i^*(e_i)$ vanishes in $H^2(P^G \setminus (\tilde{Q} \cup \tilde{L} \cup \cup_{i=1,2,3}F_i)$. We can therefore write it as the differential of a 1-form which in turn restricts to a $k_S$-linear combination of non-exact differential forms on $D_j \setminus (\tilde{Q} \cup \tilde{L} \cup \cup_{i=1,2}F_i)|_{D_j}$ for $j=\{-2,-1\}$. We get an equivalence of classes of differential forms:
\begin{equation}
\begin{split}
[(i^*(e_i),0,\ldots,0)] & = b_{i1}[(0,\theta_1,0,\ldots,0)] +\\
&+ b_{i2}[(0,0,\theta_2,0,\ldots,0)]
\end{split}
\end{equation}
where $\theta_l$ are forms of the form \eqref{diff-form-2} with simple poles at the points $u_l$ and $g_l$, for $l=1,2$.
We denote by $E_j:= H^1(D_j \setminus (\tilde{Q} \cup \tilde{L} \cup \cup_{i=1,2}F_i)|_{D_j}, D_j\cap F_k, D_j\cap F_l)$ for $j=\{-2,-1\}$ and $l,k,j$ pairwise distinct. We get an equivalence of motivic periods:
\begin{equation}
\begin{split}
\left[H_G,[\sigma_G],e_i\right]^{\mathfrak{m}} & = b_{i1}[E_1,[\sigma_G|_{E_1}],[(0,\theta_1,0,\ldots,0)]]^{\mathfrak{m}} + \\
&+ b_{i2}[E_2,[\sigma_G|_{E_2}],[(0,0,\theta_2,\ldots,0)]]^{\mathfrak{m}}
\end{split}
\end{equation}
Now notice that $[E_l,[\sigma_G|_{E_l}],[(0,\ldots,\theta_l,0,\ldots,0)]]^{\mathfrak{m}} = \log^{\mathfrak{m}}(g_lu_l|p^l_0p^l_1)$, and collect all the constants into $a_1,a_2$ to obtain the result. It is a standard exercise to put coordinates on a blow up of $\Pp^2$ at a point (see proof of lemma \ref{pullback-lemma}), which in turn enables us to compute the intersection points and the cross ratios, obtaining the result.
\end{proof}

\begin{rmk}
The constants $a_1,a_2$ depend on the choice of a basis of $(mot_G)_{\mathrm{dR}}$. In order to determine them one would need to write down such a basis, and follow the recipe in the proof of the previous theorem. 
\end{rmk}

\begin{rmk}
There are two ways of relating the motivic periods which are the conjugates of the motivic Feynman amplitude of the graph in the previous theorem to motives of its subquotient graphs. One is to consider the \textit{affine motive of a graph}, defined in \cite[5.4 and 8.5]{Brown2}. This involves removing a hyperplane for each motic subgraph of a graph, or equivalently an exceptional divisor in the blow up, in order to make the faces of $D$ affine, as was done in the proof of the previous theorem. In the example of the massless triangle the affine motive would be
\[
H^2(P \setminus (\tilde{Q} \cup \tilde{L} \cup F_1 \cup F_2 \cup F_3), D \setminus (\tilde{Q} \cup \tilde{L} \cup F_1 \cup F_2 \cup F_3))
\]
where $F_1$ and $F_2$ are as in the theorem above, and $F_3 := V(\alpha_1 + \alpha_2)$. Removing the third hyperplane is superfluous in this example because in that case we would get another motivic logarithm in the coaction, $\log^{\mathfrak{m}}\left(\frac{q_1^2}{q_2^2}\right)$, but there is an obvious relation with the two motivic logarithms in the theorem. These three motivic logarithms are periods of the affine motives of the subgraphs of the triangle graph obtained by removing one edge. However, in physics one rarely thinks of cycle-free graphs contributing logarithms.

Therefore, in order to produce a more satisfying graphical interpretation of the motivic side of the coaction in the preceding example, we should consider regularizing the motivic periods of the bubble graphs which are obtained by contracting an edge of the triangle graph. We could then obtain the conjugates in the coaction as a linear combination of these regularized motivic Feynman periods of quotient graphs. For an example  of how this works for a triangle graph with one mass vanishing in $d=2$, which is very close to Theorem 2, using tangential base point regularization see \cite[Appendix II]{Brown2}. Conjecture 1 in \cite{Brown2} predicts that after including regularized motivic Feynman periods the Galois conjugates would be motivic periods of subquotient graph motives. However, the appropriate regularization procedure for motivic Feynman amplitudes in general remains to be worked out. In addition to a generalization of the above mentioned approach in which one uses tangential base points, there is also some numerical evidence that dimensional regularization, which is the regularization method of choice in physics, could be compatible with the coaction, at least for some families of graphs -- see \cite{ABDG}. Moreover the results in \cite{ABDG} as well as \cite[\S 14]{BK} suggest that there could exist an interpretation of the de Rham periods in the Galois coaction in terms of Cutkosky cuts.

Other interesting directions for further inquiry include applying the tools presented here to some non-polylogarithmic Feynman integrals which have already been studied from a motivic point of view \cite{BV,BKV}, as well as studying the situations in which the masses and momenta lie outside of $K^{gen}_{F,G}$ in Definition \ref{space-of-kinematics-defi}, but for which there is numerical evidence that shows the Galois coaction could still be closed on the space of motivic Feynman periods, such as the ones which arise from QED \cite{S}.
\end{rmk}

\end{subsubsection}
\end{subsection}

\end{section}

Mathematical Institute, University of Oxford, OX2 6GG Oxford, UK \\
E-mail address: matija.tapuskovic@maths.ox.ac.uk
\end{document}